\documentclass[12pt]{amsart}       
\usepackage{txfonts}
\usepackage{algorithm}
\usepackage{algorithmic}
\usepackage{amssymb}
\usepackage{eucal}
\usepackage{graphicx}
\usepackage{amsmath}
\usepackage{amscd}
\usepackage[all]{xy}           
\usepackage{tikz}
\usepackage{tikz-qtree}
\usepackage{amsfonts,latexsym}
\usepackage{xspace}
\usepackage{epsfig}
\usepackage{float}
\usepackage{color}
\usepackage{fancybox}
\usepackage{colordvi}
\usepackage{multicol}
\usepackage{colordvi}
\ifpdf
  \usepackage[colorlinks,final,backref=page,hyperindex]{hyperref}
\else
  \usepackage[colorlinks,final,backref=page,hyperindex,hypertex]{hyperref}
\fi
\usepackage[active]{srcltx} 
\usepackage{mathrsfs} 

\newtheorem{theorem}{Theorem}[section]
\newtheorem{lemma}[theorem]{Lemma}

\newtheorem{corollary}[theorem]{Corollary}
\newtheorem{prop-def}{Proposition-Definition}[section]
\newtheorem{coro-def}{Corollary-Definition}[section]
\newtheorem{corollary-def}{Corollary-Definition}[section]

\newtheorem{conjecture}[theorem]{Conjecture}
\newtheorem{problem}[theorem]{Problem}
\theoremstyle{definition}
\newtheorem{defn}[theorem]{Definition}
\newtheorem{remark}[theorem]{Remark}

\newtheorem{tempex}[theorem]{Example}
\newtheorem{tempexs}[theorem]{Examples}
\newenvironment{exam}{\begin{tempex}\rm}{\end{tempex}}

\newcommand{\nc}{\newcommand}
\nc{\tred}[1]{\textcolor{red}{#1}}
\nc{\tblue}[1]{\textcolor{blue}{#1}}
\nc{\tgreen}[1]{\textcolor{green}{#1}}
\nc{\tpurple}[1]{\textcolor{purple}{#1}}
\nc{\btred}[1]{\textcolor{red}{\bf #1}}
\nc{\btblue}[1]{\textcolor{blue}{\bf #1}}
\nc{\btgreen}[1]{\textcolor{green}{\bf #1}}
\nc{\btpurple}[1]{\textcolor{purple}{\bf #1}}
\nc{\NN}{{\mathbb N}}
\nc{\ncsha}{{\mbox{\cyr X}^{\mathrm NC}}} \nc{\ncshao}{{\mbox{\cyr
X}^{\mathrm NC}_0}}

\renewcommand{\frak}{\mathfrak}

\newcommand{\efootnote}[1]{}

\renewcommand{\textbf}[1]{}

\newcommand{\delete}[1]{}

\nc{\mlabel}[1]{\label{#1}}  
\nc{\mcite}[1]{\cite{#1}}  
\nc{\mref}[1]{\ref{#1}}  
\nc{\meqref}[1]{\eqref{#1}}  
\nc{\mbibitem}[1]{\bibitem{#1}} 

\delete{
\nc{\mlabel}[1]{\label{#1}  
{\hfill \hspace{1cm}{\small\tt{{\ }\hfill(#1)}}}}
\nc{\mcite}[1]{\cite{#1}{\small{\tt{{\ }(#1)}}}}  
\nc{\mref}[1]{\ref{#1}{{\tt{{\ }(#1)}}}}  
\nc{\meqref}[1]{\eqref{#1}{{\tt{{\ }(#1)}}}}  
\nc{\mbibitem}[1]{\bibitem[\bf #1]{#1}} 
}

\nc\ns{ns\xspace}

\nc{\opo}{\mathscr{O}}
\nc{\opp}{\mathscr{P}}
\nc{\opf}{\mathscr{F}}
\nc{\optopd}{\mathscr{O}}

\nc{\name}[1]{{\bf #1}}

\nc{\tforall}{\quad \text{for all }}

\nc{\mo}{\calm}
\nc{\barot}{\overline{\otimes}}
\nc{\dm}{\diamond_\mo}
\nc{\oot}{\overline{\ot}}
\nc{\opa}{\ast} \nc{\opb}{\odot} \nc{\op}{\bullet} \nc{\pa}{\frakL}
\nc{\arr}{\rightarrow} \nc{\lu}[1]{(#1)} \nc{\mult}{\mrm{mult}}
\nc{\diff}{\mathfrak{Diff}}
\nc{\opc}{\sharp}\nc{\opd}{\natural}
\nc{\ope}{\circ}
\nc{\dpt}{\mathrm{d}}
\nc{\GS}{Gr\"obner-Shirshov\xspace}
\nc{\gsb}{Gr\"obner-Shirshov basis\xspace}
\nc{\gsbs}{Gr\"{o}bner-Shirshov bases\xspace}
\nc{\diam}{alternating\xspace}
\nc{\Diam}{Alternating\xspace}
\nc{\cdiam}{canonical alternating\xspace}
\nc{\Cdiam}{Canonical alternating\xspace}
\nc{\AW}{\mathcal{A}}
\nc{\mrbo}{modified RBO\xspace }
\nc{\ari}{\mathrm{ar}}
\nc{\lef}{\mathrm{lef}}
\nc{\Sh}{\mathrm{ST}}
\nc{\Cr}{\mathrm{Cr}}
\nc{\st}{{Schr\"oder tree}\xspace}
\nc{\sts}{{Schr\"oder trees}\xspace}
\nc{\vertset}{\Omega} 
\nc{\pb}{{\mathrm{pb}}}
\nc{\Lf}{{\mathrm{Lf}}}
\nc{\lft}{{left tree}\xspace}
\nc{\lfts}{{left trees}\xspace}
\nc{\fat}{{fundamental averaging tree}\xspace}
\nc{\fats}{{fundamental averaging trees}\xspace}
\nc{\avt}{\mathrm{Avt}}
\nc{\rass}{{\mathit{RAss}}}
\nc{\aass}{{\mathit{AAss}}}

\nc{\vin}{{\mathrm Vin}}    
\nc{\lin}{{\mathrm Lin}}    
\nc{\inv}{\mathrm{I}n}
\nc{\gensp}{V} 
\nc{\genbas}{\mathcal{V}} 
\nc{\bvp}{V_P}     
\nc{\gop}{{\,\omega\,}}     

\nc{\bin}[2]{ (_{\stackrel{\scs{#1}}{\scs{#2}}})}  
\nc{\binc}[2]{ \left (\!\! \begin{array}{c} \scs{#1}\\
    \scs{#2} \end{array}\!\! \right )}  
\nc{\bincc}[2]{  \left ( {\scs{#1} \genfrac
    \vspace{-1cm}\scs{#2}} \right )}  
\nc{\bs}{\bar{S}} \nc{\cosum}{\sqsubset} \nc{\la}{\longrightarrow}
\nc{\rar}{\rightarrow} \nc{\dar}{\downarrow} \nc{\dprod}{**}
\nc{\dap}[1]{\downarrow \rlap{$\scriptstyle{#1}$}}
\nc{\md}{\mathrm{dth}} \nc{\uap}[1]{\uparrow
\rlap{$\scriptstyle{#1}$}} \nc{\defeq}{\stackrel{\rm def}{=}}
\nc{\disp}[1]{\displaystyle{#1}} \nc{\dotcup}{\
\displaystyle{\bigcup^\bullet}\ } \nc{\gzeta}{\bar{\zeta}}
\nc{\hcm}{\ \hat{,}\ } \nc{\hts}{\hat{\otimes}}
\nc{\free}[1]{\bar{#1}}
\nc{\uni}[1]{\tilde{#1}} \nc{\hcirc}{\hat{\circ}} \nc{\lleft}{[}
\nc{\lright}{]} \nc{\lc}{\lfloor} \nc{\rc}{\rfloor}
\nc{\curlyl}{\left \{ \begin{array}{c} {} \\ {} \end{array}
    \right .  \!\!\!\!\!\!\!}
\nc{\curlyr}{ \!\!\!\!\!\!\!
    \left . \begin{array}{c} {} \\ {} \end{array}
    \right \} }
\nc{\longmid}{\left | \begin{array}{c} {} \\ {} \end{array}
    \right . \!\!\!\!\!\!\!}
\nc{\onetree}{\bullet} \nc{\ora}[1]{\stackrel{#1}{\rar}}
\nc{\ola}[1]{\stackrel{#1}{\la}}
\nc{\ot}{\otimes} \nc{\mot}{{{\boxtimes\,}}}
\nc{\otm}{\overline{\boxtimes}} \nc{\sprod}{\bullet}
\nc{\scs}[1]{\scriptstyle{#1}} \nc{\mrm}[1]{{\rm #1}}
\nc{\margin}[1]{\marginpar{\rm #1}}   
\nc{\dirlim}{\displaystyle{\lim_{\longrightarrow}}\,}
\nc{\invlim}{\displaystyle{\lim_{\longleftarrow}}\,}
\nc{\mvp}{\vspace{0.3cm}} \nc{\tk}{^{(k)}} \nc{\tp}{^\prime}
\nc{\ttp}{^{\prime\prime}} \nc{\svp}{\vspace{2cm}}
\nc{\vp}{\vspace{8cm}} \nc{\proofbegin}{\noindent{\bf Proof: }}
\nc{\proofend}{$\blacksquare$ \vspace{0.3cm}}
\nc{\modg}[1]{\!<\!\!{#1}\!\!>}
\nc{\intg}[1]{F_C(#1)} \nc{\lmodg}{\!
<\!\!} \nc{\rmodg}{\!\!>\!}
\nc{\cpi}{\widehat{\Pi}}
\nc{\sha}{{\mbox{\cyr X}}}  
\nc{\ssha}{{\mbox{\cyrs X}}} 
\nc{\shpr}{\diamond}    
\nc{\shp}{\ast} \nc{\shplus}{\shpr^+}
\nc{\shprc}{\shpr_c}    
\nc{\msh}{\ast} \nc{\zprod}{m_0} \nc{\oprod}{m_1}
\nc{\vep}{\varepsilon} \nc{\labs}{\mid\!} \nc{\rabs}{\!\mid}
\nc{\sqmon}[1]{\langle #1\rangle}

\nc{\mmbox}[1]{\mbox{\ #1\ }} \nc{\dep}{\mrm{dep}} \nc{\fp}{\mrm{FP}}
\nc{\rchar}{\mrm{char}} \nc{\End}{\mrm{End}} \nc{\Fil}{\mrm{Fil}}
\nc{\Mor}{Mor\xspace} \nc{\gmzvs}{gMZV\xspace}
\nc{\gmzv}{gMZV\xspace} \nc{\mzv}{MZV\xspace}
\nc{\mzvs}{MZVs\xspace} \nc{\Hom}{\mrm{Hom}} \nc{\id}{\mrm{id}}\nc\Idl{{\Id_{\rm Lie}}}
\nc{\im}{\mrm{im}} \nc{\incl}{\mrm{incl}} \nc{\map}{\mrm{Map}}
\nc{\mchar}{\rm char} \nc{\nz}{\rm NZ} \nc{\supp}{\rm Supp}

\nc{\Alg}{\mathbf{Alg}} \nc{\Bax}{\mathbf{Bax}} \nc{\bff}{\mathbf f}
\nc{\bfk}{{\bf k}} \nc{\bfone}{{\bf 1}} \nc{\bfx}{\mathbf x}
\nc{\bfy}{\mathbf y}
\nc{\base}[1]{\bfone^{\otimes ({#1}+1)}} 
\nc{\Cat}{\mathbf{Cat}}

\nc{\detail}{\marginpar{\bf More detail}
    \noindent{\bf Need more detail!}
    \svp}
\nc{\Int}{\mathbf{Int}} \nc{\Mon}{\mathbf{Mon}}
\nc{\rbtm}{{shuffle }} \nc{\rbto}{{Rota-Baxter }}
\nc{\remarks}{\noindent{\bf Remarks: }} \nc{\Rings}{\mathbf{Rings}}
\nc{\Sets}{\mathbf{Sets}} \nc{\wtot}{\widetilde{\odot}}
\nc{\wast}{\widetilde{\ast}} \nc{\bodot}{\bar{\odot}}
\nc{\bast}{\bar{\ast}} \nc{\hodot}[1]{\odot^{#1}}
\nc{\hast}[1]{\ast^{#1}} \nc{\mal}{\mathcal{O}}
\nc{\tet}{\tilde{\ast}} \nc{\teot}{\tilde{\odot}}
\nc{\oex}{\overline{x}} \nc{\oey}{\overline{y}}
\nc{\oez}{\overline{z}} \nc{\oef}{\overline{f}}
\nc{\oea}{\overline{a}} \nc{\oeb}{\overline{b}}
\nc{\weast}[1]{\widetilde{\ast}^{#1}}
\nc{\weodot}[1]{\widetilde{\odot}^{#1}} \nc{\hstar}[1]{\star^{#1}}
\nc{\lae}{\langle} \nc{\rae}{\rangle}
\nc{\lf}{\lfloor}
\nc{\rf}{\rfloor}

\nc{\QQ}{{\mathbb Q}}
\nc{\RR}{{\mathbb R}} \nc{\ZZ}{{\mathbb Z}}

\nc{\cala}{{\mathcal A}} \nc{\calb}{{\mathcal B}}
\nc{\calc}{{\mathcal C}}
\nc{\cald}{{\mathcal D}} \nc{\cale}{{\mathcal E}}
\nc{\calf}{{\mathcal F}} \nc{\calg}{{\mathcal G}}
\nc{\calh}{{\mathcal H}} \nc{\cali}{{\mathcal I}}
\nc{\call}{{\mathcal L}} \nc{\calm}{{\mathcal M}}
\nc{\caln}{{\mathcal N}} \nc{\calo}{{\mathcal O}}
\nc{\calp}{{\mathcal P}} \nc{\calr}{{\mathcal R}}
\nc{\cals}{{\mathcal S}} \nc{\calt}{{\mathcal T}}
\nc{\calu}{{\mathcal U}} \nc{\calw}{{\mathcal W}} \nc{\calk}{{\mathcal K}}
\nc{\calx}{{\mathcal X}} \nc{\CA}{\mathcal{A}}

\nc{\fraka}{{\mathfrak a}} \nc{\frakA}{{\mathfrak A}}
\nc{\frakb}{{\mathfrak b}} \nc{\frakB}{{\mathfrak B}}
\nc{\frakc}{{\mathfrak c}}
\nc{\frakD}{{\mathfrak D}} \nc{\frakF}{\mathfrak{F}}
\nc{\frakf}{{\mathfrak f}} \nc{\frakg}{{\mathfrak g}}
\nc{\frakH}{{\mathfrak H}} \nc{\frakI}{{\mathfrak I}}
\nc{\frakL}{{\mathfrak L}}
\nc{\frakM}{{\mathfrak M}} \nc{\bfrakM}{\overline{\frakM}}
\nc{\frakm}{{\mathfrak m}} \nc{\frakP}{{\mathfrak P}}
\nc{\frakN}{{\mathfrak N}} \nc{\frakp}{{\mathfrak p}}
\nc{\frakS}{{\mathfrak S}} \nc{\frakT}{\mathfrak{T}}

\nc{\BS}{\mathbb{S
}}

\font\cyr=wncyr10 \font\cyrs=wncyr7
\nc{\li}[1]{\textcolor{red}{#1}}
\nc{\lir}[1]{\textcolor{red}{Li: #1}}
\nc{\xing}[1]{\textcolor{blue}{Xing: #1}}
\nc{\hu}[1]{\textcolor{purple}{Huhu: #1}}
\nc{\Hu}[1]{\textcolor{purple}{#1}}
\nc{\UN}{U_{N}}
\nc{\FN}{F_{\mathrm M}}
\nc{\altx}{\Lambda}
\nc{\spr}{\cdot}
\nc{\rts}{\stackrel{\rightarrow}{\shpr}}
\nc{\ox}{\overline{\frak x}}
\nc{\oX}{\overline{X}}

\nc{\tree}{\mathbb{T}} \nc{\trees}{\mathbb{T}^\star} \nc{\treehh}{\bfktree^\star}
\nc{\treess}{\mathbb{T}^{\star_1, \star_2}}\nc{\etree}{\mathbb{1}}\nc{\bfktree}{\mathscr{T}}
\nc{\lbar}[1]{\overline{#1}} \nc{\sub}[1]{[#1]}\nc{\suba}[1]{|_{#1}} \nc{\lea}{{\rm L}}\nc{\fg}{{\rm fg}}\nc{\wt}{{\rm wt}} \nc{\degr}{{\rm deg}} \nc{\re}[1]{R(#1)}
\nc{\oid}{\mathrm{OId}} \nc{\irr}{{\rm Irr}}\nc{\irrl}{{\rm Irr_{Lie}}}   \nc{\pis}{\Pi_S} \nc{\dps}{\dotplus}
\nc{\astarrow}{\overset{\raisebox{-2pt}{{\scriptsize $\ast$}}}{\rightarrow}}
\nc{\tvarrow}[3]{#1 \overset{(t,v)}{\longrightarrow}_{#3} #2}\nc{\gs}{Gr\"{o}bner-Shirshov\xspace}
\nc{\rcp}{Rota's Program on Algebraic Operators\xspace}
\nc{\rcpo}{Rota's Program on Algebraic Operators for operads\xspace}
\nc\olie{operated Lie algebra\xspace}
\nc\olies{operated Lie algebras\xspace}
\nc\Olies{Operated Lie algebras\xspace}
\nc{\inp}{\mathrm{In}}\nc{\kirr}{\bfk\irr(S)} \nc{\rbp}{\mathscr{RBP}}
\nc\dlex{<_{dlex}} \nc\plex{<_{{\rm plex}}} \nc\pelex{\leq_{{\rm plex}}} \nc\cond{{j'<j \, \text{or}\atop  i'<i,\,j'=j}}

\nc{\compcs}{compatible compositions }
\nc{\compc}{compatible composition }
\nc\bws[1]{{\lfloor#1\rfloor}}\nc\opmx{\bfk\mathfrak{M}(X)}\nc\opmxm{\mathfrak{M}(X)}\nc\opm[1]{\mathfrak{M}(#1)}
\nc\opmz{\bfk\mathfrak{M}(Z)}\nc\opmzm{\mathfrak{M}(Z)}
\nc\Id{\rm Id}\nc\sopm[1]{\mathfrak{S}^\star(#1)}
\nc\blw[1]{\lfloor#1\rfloor}\nc\plie[1]{\mathfrak{S}(#1)}\nc\plien[1]{\mathcal{N}(#1)}\nc\nas[1]{{#1}^\ast}
\nc\ordc{>_{{\rm Dl}}} \nc\ordqc{\geq_{{\rm Dl}}}
\nc\ord{>_{{\rm IM }}}\nc\ordq{\geq _{\rm IM}}
\nc\ordd{>_{{\rm IM}}}\nc\ordqd{\geq_{{\rm IM}}}
\nc\ordb{>_{{\rm IM}}}\nc\ordqb{\geq_{{\rm IM}}}
\nc\alsw[1]{{\rm ALSW}(#1)} \nc\nlsw[1]{{\rm NLSW}(#1)}\nc\alsbw[1]{{\rm ALSBW}_{\ordq}(#1)} \nc\nlsbw[1]{{\rm NLSBW}_{\ordq}(#1)}
\nc\alsbwo[2]{{\rm ALSBW}_{#2}(#1)} \nc\nlsbwo[2]{{\rm NLSBW}_{#2}(#1)}
\nc\oplie{{\rm OLie}(X)}\nc\sopma[1]{\mathfrak{M}^\star(#1)}\nc\Pia[1]{\Pi_{#1}^{\rm ass}}\nc\Pil[1]{\Pi_{#1}^{\rm Lie}}
\nc\opliez{{\rm OLie}(Z)}\nc\opliex{{\rm OLie}(\{x,y\})}
\nc\nlsbwd[1]{\nlsw{\Delta{(#1)}}}\nc\der[2]{{#1}^{(#2)}}

\begin{document}
	
\title[Algebraic operators on Lie algebras]{Operator identities on Lie algebras, rewriting systems and Gr\"obner-Shirshov bases}

\author{Huhu Zhang}
\address{School of Mathematics and Statistics,
Lanzhou University, Lanzhou, Gansu 730000, P. R. China}
\email{zhanghh17@lzu.edu.cn}

\author{Xing Gao}
\address{School of Mathematics and Statistics,
Key Laboratory of Applied Mathematics and Complex Systems,
Lanzhou University, Lanzhou, Gansu 730000, P. R. China}
\email{gaoxing@lzu.edu.cn}

\author{Li Guo}
\address{
Department of Mathematics and Computer Science,
Rutgers University,
Newark, NJ 07102, USA}
\email{liguo@rutgers.edu}

\date{\today}
\begin{abstract}
Motivated by the pivotal role played by linear operators, many years ago Rota proposed to determine algebraic operator identities satisfied by linear operators on associative algebras, later called Rota's program on algebraic operators. Recent progresses on this program have been achieved in the contexts of operated algebra, rewriting systems and \gsbs. These developments also suggest that Rota's insight can be applied to determine operator identities on Lie algebras, and thus to put the various linear operators on Lie algebras in a uniform perspective.
This paper carries out this approach, utilizing operated polynomial Lie algebras spanned by non-associative Lyndon-Shirshov bracketed words. The Lie algebra analog of Rota's program was formulated in terms convergent rewriting systems and equivalently in terms of \gsbs. This Lie algebra analog is shown to be compatible with Rota's program for associative algebras. As applications, a classification of differential type operators and Rota-Baxter operators are presented.
\end{abstract}

\subjclass[2010]{
	05A05,   
	17B40, 
	17B38,	
	16Z10,	
	17A61,	
	16S10, 
	13P10, 
}

\keywords{\rcp; rewriting systems, Gr\"obner-Shirshov basis; operated Lie algebra; differential type operator; Rota-Baxter type operator}
\maketitle

\tableofcontents

\setcounter{section}{0}

\section{Introduction}

This paper studies operator identities on Lie algebras that are compatible with the Lie algebra structure, borrowing the insight from operator identities on associative algebras in the context of Rota's classification problem on operator identities.

\subsection{Linear operators on associative algebras and Rota's program on operator identities}

Various linear operators, characterized by the operator identities they satisfy, have played a pivotal role in mathematics research throughout the history and have attracted great interest in recent years.

The most well-known instances include
\begin{enumerate}
	\item the homomorphisms on algebras in broad areas, in particular in Galois theory, characterized by the operator identity
	$$ f(xy)=f(x)f(y);$$
	\item the derivations in analysis characterized by the Leibniz rule
	$$ d(xy)=d(x)y+xd(y);$$
	\item the integral operator characterized by the integration by parts formula
	$$ \int_a^x f'(t)g(t)\,dt = f(t)g(t)|_a^x - \int_a^x f(g)g'(t)\,dt;$$
	\item the Rota-Baxter operator that arose from probability~\mcite{Ba} and later found broad applications including renormalization of quantum field theory~\mcite{CK,Gub,Ro}, satisfying the operator identity
\begin{equation}
	P(x)P(y)=P(xP(y))+P(P(x)y)+\lambda P(xy).
	\mlabel{eq:rbo}
\end{equation}
\end{enumerate}
Some other examples are the averaging operator, Reynolds operator and the Hibert transform (also called modified Rota-Baxter operator)~\mcite{Co,Kf,Re,Tri}. See~\mcite{GGZh} for further details.

To address such linear operators uniformly, we call a linear operator {\bf algebraic} if it satisfies an algebra operator identity that is exemplified above and will be defined in general (Definition~\mref{de:algop}).

Thus algebraic operators are ubiquitous in mathematics and its applications. Motivated by this phenomenon, in his landmark paper~\mcite{Ro2}, Rota posed the question of classifying all operator identities satisfied by linear operators, that is, classifying algebraic operators.
\begin{quote}
	In a series of papers, I have tried to show that other linear operators satisfying algebraic identities may be of equal importance in studying certain algebraic phenomena, and I have posed the problem of finding all possible algebraic identities that can be satisfied by a linear operator on an algebra.
\end{quote}

We note that Rota's program aims at not only a summary of known algebraic operator identities, but all the potential algebraic operator identities that might arise in mathematical research.
Since the problem addresses a quite fundamental issue about linear operators on associative algebras, we label this problem Rota's Program on Algebraic Operators.

In a series of papers, Rota's Program on Algebraic Operators was investigated in several steps. First the term operator identities was made precise in general terms by taking them to be elements of the algebra of operated polynomial identities, as a realization of the free objects of algebras equipped by linear operators~\mcite{Guop,Ku}. Next, the operator identities similar to the ones in the above list that Rota wanted to classify were viewed as those that are compatible with the algebraic structure of the associative algebra, that is, the associativity. This viewpoint was made precise by the language of convergent rewriting systems and the method of \gsbs. After special cases of differential type and Rota-Baxter type were treated carefully in~\mcite{GSZ,ZGGS}, a general setup was obtained in~\mcite{GG} with applications to modified Rota-Baxter operators.

This program has two advantages in the study of linear operators. One is that it gives a uniform approach to classes of operators, such the differential type operators and Rota-Baxter type operators, instead of dealing with one operator at a time. For example, free objects of the two types of operators were constructed simultaneously in~\mcite{GSZ,ZGGS}. The second advantage is that the program gives characterizations of the interesting operators, which might lead to the discovery of new operators and further to their classification. For example, the notions of differential type and Rota-Baxter type operators give rise to new operators in \mcite{GSZ,ZGGS}. As seen in this paper, these advantages carry over for linear operators on Lie algebras.

\subsection{Linear operators on Lie algebras}
As in the case of associative algebras, linear operators on Lie algebras have also played an important role in several areas in pure and applied mathematics.

One distinct instance is the Rota-Baxter operator on Lie algebra defined by Eq.~\meqref{eq:rbo} on Lie algebras. It first appeared as the operator form of the classical Yang-Baxter equation and served as a fundamental tool in integrable systems~\mcite{BGN,STS,RS1,RS2}.
Also appearing in the same context is the modified Rota-Baxter operator satisfying the modified classical Yang-Baxter equation~\mcite{Bo,Kup,STS}
$$ [P(x),P(y)]=P([P(x),P])+P([x,P(y)])-[x,y].$$

Another example is the Nijenhuis operator from pseudo-complex manifolds, differential geometry, deformation theory and integrable systems~\mcite{FN,KM,N}.
Here the operator identity is
\begin{equation}
	[N(x),N(y)] = N([N(x),y]) + N([x,N(y)]) - N^2([x,y]).
	\mlabel{eq:Nij}
\end{equation}
The study of Nijenhuis operator on an associative algebra was more recent, first introduced by Carinena and coauthors~\mcite{CGM} to study quantum bi-Hamiltonian systems.
In~\mcite{Uc}, Nijenhuis operators are constructed by analogy with Poisson-Nijenhuis geometry, from relative Rota-Baxter operators.

In recent years, other linear operators such as differential operators, Reynolds operators, Rota-Baxter families and matching Rota-Baxter algebras have been studied for Lie algebras, in the directions of deformations, cohomology and bialgebras~\mcite{Das,GGZ,GLS,LST,TBGS,ZG,ZyGM}. Thus it is natural to adapt the program of Rota on algebraic operators from associative algebras to Lie algebras. This is the purpose of this paper.

\subsection{Rota's program for Lie algebras and an outline of the paper}
In order to adapt the approach to Rota's program of algebraic operators on associative algebras, via operated algebras, rewriting systems and \gsbs, we first regard an operator identity on Lie algebras as an element of the operated Lie polynomial algebra. The latter is given as a realization of the free operated Lie algebras, first obtained in~\mcite{QC} in terms of non-associative Lyndon-Shirshov bracketed words with respect to a particular monomial order. For later applications, we give an axiomatic approach based on the notion of an invariant monomial order (Theorem~\mref{lem:freelie}). This is presented in Section~\mref{ssec:oplie}.

Then in Section~\mref{sec:lie}, we formulate Rota's program on algebraic operators for Lie algebras in the context of rewriting systems and \gsbs (Problems~\mref{prob:rpcltrs} and \mref{prob:rpclgsb}). The equivalence of the two formulations is established (Theorem~\mref{lem:liegreq}). Furthermore, Rota's program for associative algebras is shown to be compatible with Rota's program for Lie algebras (Theorem~\mref{thm:alrcpeq}), allowing us to pass information between the associative case and the Lie algebra case and to explain the phenomenon that operators on associative algebras also appeared for Lie algebras.

As applications of the general results on Rota's program, in particular the connection between the associative and Lie algebra cases, in Section~\mref{sec:app} we focus on the special cases when the operator identities resemble the differential operator and Rota-Baxter operator respectively, allowing us to identify the operator identities that give convergent rewriting systems and \gsbs (Theorems~\mref{thm:gsbdt} and \mref{thm:gsbrbt}). Going beyond these types of operators, we also show that the modified Rota-Baxter operatorhas its operator identity (modified Yang-Baxter equation) giving rise to a convergent rewriting system and a \gsb (Theorem~\mref{thm:modrbDL}).

\section{\Olies}
\mlabel{ssec:oplie}
To formulate \rcp for Lie algebras, we construct free operated Lie algebras.
Before analyzing the related topics, we first present the construction of free \olies (also called Lie $\Omega$-algebras). The construction was first obtained in~\mcite{QC} with respect to a particular monomial order on bracketed words. Here we extend the construction with respect any invariant monomial order, giving us the flexibility for various applications. 

\subsection{Associative and non-associative Lyndon-Shirshov words }
\mlabel{ssec:aslsw}
For a set $Y$, let $S(Y)$ (resp. $M(Y)$) denote the free semigroup (resp. free monoid) on $Y$, consisting of words (resp. words including the empty word $\bfone$) in the alphabet set $Y$.
Let $\nas{Y}$ denote the free magma on $Y$, consisting of non-associative (binary) words on $Y$. Such a word is either in $Y$ or is of the form $(w)=((u)(v))$ for words $(u), (v)\in \nas{Y}$. These words can be realized as planar binary trees~\mcite{Reu}.

It is well known that non-associative Lyndon(-Shirshov) words form a linear basis for a free Lie algebra~\mcite{BC,Shi}.
Let $(X, \geq)$ be a well-ordered set. Define the lex-order $\geq_{\rm lex}$ on the free monoid $M(X)$ over $X$ by
\begin{enumerate}
	\item $\bfone>_{\rm lex} u$ for all nonempty word $u$, and
	\item for any $u = xu'$ and  $v = yv'$ with $x, y\in X$,
$$ u>_{\rm lex} v\,\text{ if }\,  x > y, \text{ or }\, x = y\,\text{ and }\, u'>_{\rm lex} v'.$$
\end{enumerate}
For example, let $x>y$. Then $\bfone>x>xx>xy>\cdots>y>yx>yy.$
\begin{defn} Let $(X, \geq)$ be a well-ordered set.
\begin{enumerate}
\item An associative word $w\in S(X)$ is called an {\bf associative Lyndon-Shirshov word} on $X$ with respect to the lex-order $\geq_{\rm lex}$, if $w =uv >_{\rm lex}vu$ for every decomposition $w = uv$ of $w$ with $u,v\in S(X)$.
\item A non-associative word $(w)\in \nas{X}$  is called a {\bf non-associative Lyndon-Shirshov word} on $X$ with respect to the lex-order $\geq_{\rm lex}$, provided
\begin{enumerate}
  \item the corresponding associative word $w$ is an associative Lyndon-Shirshov word on $X$;
  \item if $(w) = ((u)(v))$, then both $(u)$ and $(v)$ are non-associative Lyndon-Shirshov words on $X$;
  \item if $(w) = ((u)(v))$ and $(u) = ((u_1 )(u_2 ))$, then $v \geq_{\rm lex} u_2.$
\end{enumerate}
\end{enumerate}
\end{defn}

\begin{remark}
An equivalent condition for $w$ to be an associative Lyndon-Shirshov word is that,
for any decomposition of $w = uv$ with $u,v\in S(X)$,
$w>_{\rm lex} v.$
\end{remark}

\begin{exam}
Let $X=\{x,y,z\}$ with $x>y>z$. Then
\begin{enumerate}
  \item The words $xy$, $xyz$ and $xzy$ are associative Lyndon-Shirshov words, but the words $yx$, $yxz$, $yzx$, $zxy$ and $zyx$ are not associative Lyndon-Shirshov words.
  \item The words $(xy)$, $(x(yz))$ and $((xz)y)$ are non-associative Lyndon-Shirshov words, but $(yx)$ and $((xy)z)$ are not non-associative Lyndon-Shirshov words.
\end{enumerate}
\end{exam}

Denote by $\alsw{X}$ (resp. $\nlsw{X}$) the set of all associative (resp. non-associative) Lyndon-Shirshov words on a well-ordered set $X$ with respect to the lex-order $\geq_{\rm lex}$. For any $w\in\alsw{X}$, there exists a unique
procedure, called the {\bf Shirshov standard bracketing}~\mcite{BC,Shi}, to give a non-associative Lyndon-Shirshov word $[w]$. Furthermore,
\begin{equation} \nlsw{X} =\left\{[w]\,|\,w \in\alsw{X}\right\}.
	\mlabel{eq:ssbw}
\end{equation}
It is obtained by a recursion in which each step brackets the minimal letter in the word to the previous letter to give a new letter, and then proceeds to the next step, until there is only one letter left.

We use the following example as an illustration and refer the reader to the original literature for further details.

Take $w=xxyyxy \in \alsw{X}$ on $X = \{x,y\}$ with $x>y$, we obtain $[w]=((x((xy)y))(xy))\in\nlsw{X}$ after the following recursion.
\begin{enumerate}
\item  Bracketing the minimal letter $y$ in $w$ to the previous letters ($\neq y$) to form a new letter, we obtain a new associative Lyndon-Shirshov word $x(xy)y(xy)$ on the new letter set $\{x,(xy),y\}$ with $x>_{\rm lex} (xy) >_{\rm lex}y.$
\item  Bracketing the minimal letter $y$ to the previous letters, we get a new associative Lyndon-Shirshov word $x((xy)y)(xy)$ on the new letter set $\{x,(xy), ((xy)y)\}$ with $x>_{\rm lex}(xy)>_{\rm lex}((xy)y)$.
\item  Bracketing the minimal letter $((xy)y)$ to the previous letters, we have a new associative Lyndon-Shirshov word $(x((xy)y))(xy)$ on the new letter set $\{(x((xy)y)),(xy)\}$ with $(x((xy)y))>_{\rm lex}(xy).$
\item  Bracketing the minimal letter $(xy)$ to the previous letters, we have a new associative Lyndon-Shirshov word $((x((xy)y))(xy))$ on the single letter $\{((x((xy)y))(xy))\}$.
\item  Then $[w]=((x((xy)y))(xy))$ is in $\nlsw{X}$.
\end{enumerate}

Let $X$ be a well-ordered set. Let Lie$(X)$ be the Lie subalgebra of commutator Lie algebra $\big(\bfk \langle X\rangle, (-,-)\big)$ generated by $X$, where $\bfk \langle X\rangle=\bfk  M(X)$ is the free associative algebra on $X$.
As is well known, Lie$(X)$ is a free Lie algebra on the set $X$, with a linear basis given by the non-asssociative Lyndon-Shirshov words $\nlsw{X}$~\mcite{Reu}.

\subsection{Associative and non-associative Lyndon-Shirshov bracketed words }
\mlabel{ssec:anlsw}
In this subsection, we summary the construction of the free \olies in ~\mcite{QC}. An operated version of non-associative Lyndon-Shirshov words will be a linear basis of the free \olie.

For any set $Y$,  let
$$\blw{Y}:=\{\blw{y}\,|\,y\in Y\}$$
denote a copy of $Y$ that is disjoint from $Y$.

Now fix a set $X$. We construct the sets of associative and non-associative bracketed words on $X$ by direct systems $\plie{X}_n$ and $\plien{X}_n, n\geq 0$ recursively defined as follows. First denote
$$\plie{X}_0:=S(X)  \text{  and } \plien{X}_0:=\nas{X}.$$
Next for $k\geq 0$, assume that $\plie{X}_{n}$ and $\plien{X}_{n}$ have been defined for $n< k$, such that
\begin{enumerate}
	\item $\plie{X}_{n+1}:= S(X\cup \blw{\plie{X}_{n}})\,\text{ and }\, \plien{X}_{n+1}:= \nas{(X\cup \blw{\plien{X}_{n}})};$
	\item there are natural injections $\plie{X}_{n}\hookrightarrow\plie{X}_{n+1}$ and $\plien{X}_{n}\hookrightarrow\plien{X}_{n+1}$ of semigroups and magmas.
\end{enumerate}
\delete{ together with
natural embeddings $\plie{X}_{n}\hookrightarrow\plie{X}_{n+1}$ and $\plien{X}_{n}\hookrightarrow\plien{X}_{n+1}$ for $n<k.$}
We then recursively define
$$\plie{X}_{k+1}:= S(X\cup \blw{\plie{X}_{k}})\,\text{ and }\, \plien{X}_{k+1}:= \nas{(X\cup \blw{\plien{X}_{k}})}.$$
Further the injections  $\plie{X}_{k-1}\hookrightarrow\plie{X}_{k}$ and $\plien{X}_{k-1}\hookrightarrow\plien{X}_{k}$ induce set injections
$$X\cup \lc\plie{X}_{k-1}\rc\hookrightarrow X \cup \lc\plie{X}_{k}\rc, \quad X\cup \lc\plien{X}_{k-1}\rc\hookrightarrow X\cup \lc\plien{X}_{k}\rc.$$
Then the functoriality of taking the free semigroup and free magma extends these set injections to the injections
$$\plie{X}_{k}\hookrightarrow\plie{X}_{k+1}\,\text{ and }\, \plien{X}_{k}\hookrightarrow\plien{X}_{k+1}$$
of free semigroups and free magmas.

Finally, define the direct limits
$$\plie{X}:=\bigcup_{n\geq0}\plie{X}_{n}\,\text{ and }\, \plien{X}:=\bigcup_{n\geq0}\plien{X}_{n}.$$
Elements of $\plie{X}$ (resp. $\plien{X}$) are called the {\bf associative  (resp. non-associative) bracketed words on $X$.}
We also denote an element of $\plie{X}$ (resp. $\plien{X}$) by $w$ (resp. $(w)$).

We also recall some basic properties and notations of bracketed words on $X$.

\begin{enumerate}
	\item Every element $w$ of $\plie{X}$ can be uniquely written in the form
	$$w=w_1\cdots w_k,$$
	for $w_i\in X\cup \blw{\plie{X}}$, $ 1 \leq i \leq k, k\geq 1$. We call $w_i$ {\bf prime} and $k$ the {\bf breadth} of $w$, denoted by $|w|$.
	
	\item Define the {\bf depth} of $w\in\plie{X}$ to be ${\rm dep(w)} := {\rm min}\{n\,|\,w\in\plie{X}_n\}.$
	
	\item For any $(w)\in\plien{X}$, there exists a unique $w\in\plie{X}$ by forgetting the brackets of $(w)$. Then we can define the {\bf depth} of $(w)\in\plien{X}$ to be ${\rm dep((w))} := {\rm dep(w)}.$ This agrees with $\min \{n\,|\, (w)\in \plien{X}_n\}$.
	
	\item The {\bf degree} of $w\in\plie{X}$, denoted by ${\rm deg}(w)$, is defined to be the total number of occurrences of all $x\in X$ and $\blw{~}$ in $w$.
\end{enumerate}

For example, if $w=\blw{xy\blw{z}y}xy\in\plie{X}$ with $x,y,z\in X$, then
$$|w|=3, \, {\rm dep}(w)=2\,\text{ and }\, {\rm deg}(w)=8.$$
If $(w)=x((\blw{y}(xy))z)\in \plien{X}$ with $x,y,z\in X$, then $$w=x\blw{y}xyz\in \plie{X}\,\text{ and }\, {\rm dep}((w))={\rm dep}(w)=1.$$

\begin{defn}~\mcite{QC}
\begin{enumerate}
\item An {\bf \olie} is a Lie algebra $L$ together with a linear map
$P_L: L\to L$.
\item A {\bf morphism} from an \olie $(L_1,P_{L_1})$ to an \olie $(L_2,P_{L_2})$
is a Lie algebra homomorphism $f : L_1\to L_2$ such that
$f\circ P_{L_1}=P_{L_2}\circ f.$
\end{enumerate}
\end{defn}

We also recall the notion of star words. See~\mcite{BC,GSZ,GG} for details.
\begin{defn}
Let $\star$ a symbol not in a set $X$ and let $X^\star=X\sqcup \{\star\}$.
\begin{enumerate}
  \item  A word in $\plie{X^\star}$ is called a {\bf $\star$-bracketed word} on $X$ if $\star$ appears only once.
  The set of all $\star$-bracketed words on $X$ is denoted by $\sopm{X}$.
  \item  For $q\in \sopm{X}$ and $u \in\plie{X}$, we define $q\suba{ u}$ (or $q\suba{\star\mapsto u}$) to be the bracketed word on $X$
obtained by replacing the symbol $\star$ in $q$ by $u$.
  \item  For $q\in \sopm{X}$ and $s =\sum_i c_i u_i \in\bfk\plie{X}$, where $c_i\in\bfk$ and $u_i\in\plie{X}$, we define
$q\suba{s}:=\sum_i c_i q\suba{u_i}.$
  \item A bracketed word $u \in\plie{X}$ is a {\bf subword} of another bracketed word $ w\in\plie{X}$ if
$w = q\suba{u}$ for some $q \in\sopm{X}$.
\end{enumerate}
\end{defn}
 A monomial order is a well-order that is compatible with all operations in the algebraic structure. This can be given concisely by $\star$-bracketed words.

\begin{defn}
Let $X$ be a set. A {\bf monomial order} on $\plie{X}$ is a well order $\geq$ on $\plie{X}$ such that
\begin{equation}
\quad  u >  v \Rightarrow  q\suba{u} >  q\suba{v},\,\text{ for all } u,v\in \plie{X} \text{ and } q\in\sopm{X}.
\mlabel{eq:mono}
\end{equation}
\end{defn}

Let $\geq$ be a monomial order on $\plie{X}$.
Given a bracketed word $f\in \bfk\plie{X}$, we let $\lbar{f}$ denote the {\bf leading bracketed word} (monomial) of $f$.
We call $f$ {\bf monic} with respect to $\geq$ if the coefficient of $\lbar{f}$ is 1.
A subset $S\subseteq \plie{X}$ is called {\bf monic} if each element of $S$ is monic.

We introduce the following notion for the construction of free operated Lie algebras. 
\begin{defn}
An order, in particular a monomial order, $\geq$ on $\plie{X}$ is called {\bf invariant} if, for all prime elements $u_1, \ldots, u_n \in \plie{X}$ and $\sigma\in S_n$, we have
\begin{equation}
	u_1 \cdots u_n  \geq u_{\sigma(1)} \cdots u_{\sigma(n)} \Longleftrightarrow u_1 \cdots u_n  \succeq_{\rm lex} u_{\sigma(1)} \cdots u_{\sigma(n)},
	\mlabel{eq:dlor}
\end{equation}
where $\succeq$ is the restriction of $\geq$ to the set of prime elements $X\sqcup\blw{\plie{X}}$. To emphasize, we let $\ordq$ denote an invariant monomial order on $\plie{X}$.
\mlabel{de:inv}
\end{defn}

We display two examples of invariant monomial orders on $\plie{X}$ for later applications. 

\begin{exam}
Let $(X, \geq)$ be a well-ordered set.
Take $u=u_1\cdots u_m$ and $v=v_1\cdots v_n$ in $\plie{X}$, where $u_i$ and $v_j$ are prime. Define
$u\ordc v $ inductively on $\dep(u)+\dep(v)\geq 0$.
For the initial step of $\dep(u)+\dep(v) = 0$, we have $u,v\in S(X)$ and define $u\ordc v$ by
$u >_{\rm deg-lex} v$, that is, $$u \ordc v \, \text{ if }\, ({\rm deg}(u), |u|,  u_1, \ldots, u_m) >({\rm deg}(v), |v|,  v_1, \ldots, v_n) \, \text{ lexicographically}.$$
Here notice that ${\rm deg}(u)= |u|$ and ${\rm deg}(v)= |v|$.
For the induction step, if $u = \lc u'\rc$ and $v = \lc v' \rc$, then define
$u\ordc v \,\text{ if }\, u' > v'.$
Otherwise, define
$$u\ordc v \, \text{ if }\, ({\rm deg}(u), |u|,  u_1, \ldots, u_m) >({\rm deg}(v), |v|,  v_1, \ldots, v_n) \, \text{ lexicographically}.$$
Then $\ordqc$ is a monomial order~\cite{QC} and is invariant. 
\mlabel{ex:inva1}
\end{exam}

\begin{exam}
Let $(X, \geq)$ be a well-ordered set.  Denote by $\deg_X(u)$ the
number of $x \in X$ in $u$ with repetition. Define the order $\geq_{\rm dt}$ on $\plie{X}$ as follows. For any $u=u_1\cdots u_m$ and $v=v_1\cdots v_n$, where $u_i$ and $v_j$ are prime. Define
$u>_{\rm dt} v $ inductively on $\dep(u)+\dep(v)\geq 0$.
For the initial step of $\dep(u)+\dep(v) = 0$, we have $u,v\in S(X)$ and define $u>_{\rm dt} v$ if $u>_{\rm deg-lex} v$, that is
$$(\deg_X(u), u_1, \ldots, u_m) > (\deg_X(v), v_1, \ldots, v_n) \, \text{ lexicographically}.$$
For the induction step, if $u = \lc u'\rc$ and $v = \lc v' \rc$, define
$$u>_{\rm dt} v \,\text{ if }\,  {u'} >_{\rm dt} {v'}.$$
If $u = \lc u'\rc$ and $v\in X$, define $u >_{\rm dt} v$.
Otherwise, define
$$u>_{\rm dt} v \, \text{ if }\, (\deg_X(u),  u_1, \ldots, u_m) >(\deg_X(v),  v_1, \ldots, v_n) \, \text{ lexicographically}.$$
Then $\geq_{\rm dt}$ is a monomial order~\cite{GSZ}, which is also invariant. 
\mlabel{ex:inva2}
\end{exam}

To formulate \rcp for Lie algebras, we fix an invariant monomial order $\ordq$ on $\plie{X}$ and construct a free operated Lie algebra on $X$. It is defined by a direct system of bracketed words from a recursion. First define
$$\alsbw{X}_0:=\alsw{{X}}$$
and then
$$\nlsbw{X}_0:=\nlsw{{X}}=\{[w]\,|\,w\in \alsbw{X}_0\}$$
with respect to the order $\succeq_{{\rm lex}}$.
Here $[w]$ is obtained by using the Shirshov bracketing for $w$. For the recursion step, for any given $n\geq 0$, assume that we have defined
$$\alsbw{X}_{n}$$
and then
$$\nlsbw{X}_{n}=\{[w]\,|\,w\in \alsbw{X}_{n}\}$$
with respect to the order $\succeq_{{\rm lex}}$.
Then we first define
$$\alsbw{X}_{n+1}:=\alsw{X\sqcup \blw{\alsbw{X}_{n}}}$$
with respect to the order $\succeq_{{\rm lex}}$.
We then denote
$$\big[X\sqcup \blw{\alsbw{X}_{n}}\big]:=\left\{[w]:=
\left\{\left .
\begin{array}{ll}
w, & \hbox{if  $w\in X$}, \\
\blw{[w']}, & \hbox{if $w=\blw{w'}$}
\end{array}
\right.\,\right|\, w \in X\sqcup \blw{\alsbw{X}_{n}}\right\}$$
and define
\begin{equation}
\nlsbw{X}_{n+1}:=\nlsw{\big[X\sqcup \blw{\alsbw{X}_{n}}\big]}=\{[w]\,|\,w\in\alsbw{X}_{n+1}\}
\mlabel{eq:bnlsbw}
\end{equation}
with respect to the order $\succeq_{\rm lex}$.
We finally denote
$$\alsbw{X}:=\bigcup_{n\geq0}\alsbw{X}_n\,\text{ and }\,\nlsbw{X}:=\bigcup_{n\geq0}\nlsbw{X}_n.$$
Then we have
\begin{equation}
\nlsbw{X}=\big\{[w]\,|\,w\in\alsbw{X}\big\}.
\mlabel{eq:astolie}
\end{equation}
Elements of $\alsbw{X}$ (resp. $\nlsbw{X}$) are called the {\bf associative (resp. non-associative) Lyndon-Shirshov bracketed words or Lyndon-Shirshov $\Omega$-words} with respect to $\ordq$.

It follows from Eq.~(\mref{eq:astolie}) that any associative Lyndon-Shirshov bracketed word $w$ is associated to a unique non-associative Lyndon-Shirshov bracketed word $[w]$. For $[u], [v] \in \nlsbw{X}$, define $[u]\geq [v]$ if and only if $u\geq v$. Note that since $[u]\mapsto u$ is bijective, we can compare $[u], [v]$ by comparing $u,v$. So the orders
are essentially on associative words, and non-associative (Lyndon-Shirshov bracketed) words are compared by the corresponding associative words.

\begin{exam}
Suppose that $w=\blw{xyz}\bws{x}\bws{y}\in \alsbwo{X}{\ordq}$ with $x\ord y\ord z\in X$ and $\blw{xyz}\ord\bws{x}\ord\bws{y}.$ Then $\blw{xyz}\succ\bws{x}\succ\bws{y}$ and so $\blw{(x(yz))}\succ\bws{x}\succ\bws{y}$. It follows from Eq.~\meqref{eq:bnlsbw} that
      $$[w]=[[\blw{xyz}][\bws{x}][\bws{y}]]=[\blw{[xyz]}\bws{[x]}\bws{[y]}] = [\blw{(x(yz))}\bws{x}\bws{y}].$$
Further for the word $\blw{(x(yz))}\bws{x}\bws{y}$ on
$\{\blw{(x(yz))}, \bws{x}, \bws{y}\}$, applying the Shirshov standard bracketing, we obtain
$[w] = \big( \blw{(x(yz))}(\bws{x}\bws{y})\big).$
\end{exam}

Denote by $\oplie$ the operated Lie subalgebra of $\bfk\plie{X}$ generated by $X$ under the Lie bracket $(u,v) = uv - vu$. We show that it is a free operated Lie algebra on $X$. 

\begin{lemma}
Let $u\in \alsbw{X}$ and $[u]\in\nlsbw{X}$. Considering $[u]$ as a polynomial in $\bfk\plie{X}$, then $\lbar{[u]}=u$ with respect to the invariant monomial order $\ordq$ on $\plie{X}$.
\mlabel{lem:ldt}
\end{lemma}
\begin{proof}
We use induction on ${\rm dep(u)}\geq 0$.
If ${\rm dep(u)}=0$, the result holds by~\cite[Lemma 2.2]{QC}.
Assume the result is true for $u\in \alsbw{X}$ with ${\rm dep(u)}\leq n-1.$
Depending on the breadth $|u|\geq 1$, we have the following two cases to consider.

\noindent {\bf Case 1.} $|u|=1$. Then $u=\blw{u'}$ for some $u'\in\alsbw{X}$ and so
$[u]=\blw{[u']}$. Hence
  $$\lbar{[u]}=\lbar{\blw{[u']}}=\blw{\lbar{[u']}}=\blw{u'}=u.$$
Here the second step following from that $\ordq$ is a monomial order and
the third step employs the induction hypothesis.

\noindent {\bf Case 2.} $|u|>1$. Then we can write $u=u_1\cdots u_k$ for some $k>1$ and $u_i\in X\cup\blw{\plie{X}}.$
By Eq.~\meqref{eq:bnlsbw}, $[u]=[[u_1]\cdots [u_k]]$. Further by~\cite[Lemma 2.2]{QC},
$$[u]=[[u_1]\cdots [u_k]]=[u_1]\cdots [u_k]+\sum_{\sigma\in S_k\setminus \{\id\}} k_\sigma[u_{\sigma(1)}]\cdots [u_{\sigma(k)}].$$
Since $[u]$ as a non-associative Lyndon-Shirshov word on the letter set $\{[u_1],\ldots, [u_k]\}$, we have
$$[u_1]\cdots [u_k] \succ_{\rm lex} [u_{\sigma(1)}]\cdots [u_{\sigma(k)}]\,\text{ for any }\,\sigma\in S_k\setminus \{\id\},$$
and by Eq.~(\mref{eq:dlor}),
$$[u_1]\cdots [u_k] \ord [u_{\sigma(1)}]\cdots [u_{\sigma(k)}]\,\text{ for any }\,\sigma\in S_k\setminus \{\id\}.$$
Finally by Case 1 and the property of monomial order,
$$\lbar{[u]}=\lbar{[u_1]\cdots [u_k]}=\lbar{[u_1]}\cdots \lbar{[u_k]}=u_1\cdots u_k=u.$$
This completes the proof.
\end{proof}

\begin{lemma}
	Let $X$ be a well-ordered set. For any $(u)\in\plien{X}$, it has a representation
	$(u)=\sum_i \alpha_i[u_i]$ in $\opliez$, where each $\alpha_i\in\bfk$ and $[u_i]\in \nlsbw{X}$.
	\mlabel{lem:nassb}
\end{lemma}
\begin{proof}
We first employ induction on ${\rm dep(u)}\geq 0$.
The initial step of ${\rm dep(u)}=0$ follows from~\cite[Lemma 2.1]{QC}.
For the induction step, we reduce the proof to the induction on the breadth $|u|\geq 1$. If $|u|=1$, then $u=\blw{u'}$ for some $u'\in\plie{X}$ and so
	$(u)=\blw{(u')}$. By the induction hypothesis on depth,
	$(u')=\sum_i \alpha_i[{u}_i']$ for some $\alpha_i\in \bfk$ and $[{u}_i']\in\nlsbw{X}$. Hence
	$$(u)=\blw{(u')}=\sum_i \alpha_i\blw{[{u}_i']}\in \opliez.$$
	If $|u|>1$,  then $(u)=((w_1)(w_2))$ for $w_1, w_2\in\plie{X}$. By the induction hypothesis on breadth,
	we have
	$$(w_1)=\sum_k \beta_k[{w}_{1k}]\,\text{ and }\, (w_2)=\sum_{\ell} \gamma_{\ell}[{w}_{2{\ell}}].$$
	Hence
	$$(u)=((w_1)(w_2))=\Big(\Big(\sum_k \beta_k[{w}_{1k}]\Big)\Big(\sum_{\ell} \gamma_{\ell}[{w}_{2{\ell}}]\Big)\Big)=\sum_{k,{\ell}}\beta_k\gamma_{\ell}([{w}_{1k}][{w}_{2{\ell}}]),$$
	where each word $([{w}_{1k}][{w}_{2\ell}])$ is a non-associative word on the letter set $\{[w_{11}],[w_{21}],\ldots, [w_{1k}],[w_{2k}],\ldots\}$.
	By~\cite[Lemma 2.1]{QC},
	$$(u)=\sum_{k,{\ell}}\beta_k\gamma_{\ell}([{w}_{1k}][{w}_{2\ell}])=\sum_i \alpha_i[u_i],$$
	for some $\alpha_i\in \bfk$ and $[u_i]\in{\rm NLSW(\{[w_{11}],[w_{21}],\ldots, [w_{1k}],[w_{2k}],\ldots\})}\subseteq \nlsbw{X}$.
\end{proof}

For example, let $(u)=\big((\bws{x}\bws{y})\bws{z}\big)$ with $x>y>z$. Then $(u)=\big(\bws{x}(\bws{y}\bws{z})\big)+\big((\bws{x}\bws{z})\bws{y}\big)$, where $\big(\bws{x}(\bws{y}\bws{z})\big)$ and $\big((\bws{x}\bws{z})\bws{y}\big)$ are two non-associative Lyndon-Shirshov bracketed words.

\begin{theorem}
For a given invariant monomial order $\ordq$ on $\plie{X}$, the operated Lie algebra $\oplie$, together with the natural embedding $X\to \oplie$,
is a free \olie on $X$ with a linear basis $\nlsbw{X}$.
\mlabel{lem:freelie}
\end{theorem}

\begin{proof}
Since $\oplie$ is a free Lie algebra on $X$, it suffices  to prove that $\nlsbw{X}$ is a linear basis of $\oplie$.
By Lemma~\mref{lem:nassb}, $\nlsbw{X}$ is a linear generating set of $\oplie$.
Suppose
$$\sum_{i=1}^n k_i [u_i]=0,\, \text{ where }\, k_i\in\bfk\,\text{ and }\, [u_i]\in\nlsbw{X}.$$
We are left to show $k_1 = \cdots =k_n=0$, which will be done by induction $n\geq 1$. If $n=1$, then $k_1 = 0$. Consider $n\geq 2$. Without loss of generality, we may assume $u_1\ord \cdots\ord u_n.$ If $k_1\neq0$, then by Lemma~\mref{lem:ldt},
$$\lbar{\sum_{i=1}^n k_i [u_i]} = \lbar{[u_1]} = u_1,$$
a contradiction. So $k_1=0 $ and by the induction hypothesis, $k_2 =\cdots =k_n = 0$, as required.
\end{proof}

Thanks to the above theorem, we propose the following notion.
\begin{defn}
Elements of $\oplie$ are called {\bf operated Lie polynomials}.
For $\phi\in\oplie$, we call $\phi=0$ (or simply $\phi$) an {\bf operated Lie polynomial identity (OLPI)}. When $X=\{x_1,\ldots,x_n\}$, we denote $\phi\in \oplie$ by $\phi(x_1,\ldots,x_n)$.
\end{defn}

Given an \olie $(R,P)$ and a map
$$\theta: X\to R,$$
the universal property of the free \olie $\oplie$ gives a unique morphism
$$\tilde{\theta}:\oplie\to R$$
of \olies extending $\theta$.
We call this morphism the {\bf evaluation map} at $\theta$ and
$$\phi(r_1, \ldots ,r_k):=\tilde{\theta}(\phi(x_1, \ldots ,x_k))$$
the {\bf evaluation} of $\phi\in \oplie$ at $\theta$.

\begin{defn}
Let $\phi\in \oplie$ and $(R,P)$ be an \olie.
\begin{enumerate}
  \item We say that
$R$ is a {\bf $\phi$-Lie algebra} and $P$ is a {\bf $\phi$-operator}, if
$$\phi(r_1, \ldots ,r_k)=0,\,\forall r_1, \ldots ,r_k\in R.$$
More generally, for a subset $\Phi\subseteq\oplie$, we say that $R$ (resp. $P$) a {\bf $\Phi$-Lie algebra} (resp. {\bf $\Phi$-operator}) if $R$ (resp. $P$) is a $\phi$-Lie algebra (resp. $\phi$-operator) for every $\phi\in\Phi$.
\item The linear operator $P$ on the Lie algebra $R$ is called {\bf (Lie) algebraic} if there is $0\neq \Phi\subseteq \oplie$ such that $P$ is a $\Phi$-operator.
\mlabel{de:lalgop}
\end{enumerate}
\end{defn}

For example, for $\phi(x,y) = [\bws{xy}] - [\bws{x}y] - [x\bws{y}]\in\oplie$, a $\phi$-algebra is just a differential Lie algebra. Also, $[\bws{x}\bws{y}]-[\bws{\bws{x}y}]-[\bws{x\bws{y}}]+[xy]$ defines the (operator form of the) modified Yang-Baxter equation on a Lie algebra~\mcite{STS,ZGG1,ZGG2,RS1, RS2}.

\section{\rcp for Lie algebras}
\mlabel{sec:lie}

\rcp is to classify linear operator identities on associative algebras and has been treated in~\mcite{GG,GSZ,ZGGS}. As in the case of associative algebras, linear operators on Lie algebras have also played a pivot role in their studies and applications. This motivates us to adapt the study of \rcp to the context of Lie algebras. This is obtained in this section, utilizing the structure of operated Lie algebra polynomials presented in the previous section.

\subsection{Special normal words}\mlabel{starbw}

The following result plays a key role in defining the rewriting systems and \gsbs for operated Lie algebras.
\begin{lemma}
Let $Z$ be a well-ordered set and let $q\in\sopm{Z}$ and $u$, $q\suba{u}\in\alsbw{Z}$. Then
$[q\suba{u}]=[q'\suba{uc}]$
for some $q'\in\sopm{Z}$ and $c\in\plie{Z}\cup\{\bfone\}$. Let
\begin{equation}
[q\suba{u}]_u:=[q'\suba{[uc]}]_{[uc]\mapsto [\cdots[[[u][c_1]][c_2]]\cdots[c_m]]}=[q'\suba{[\cdots[[[u][c_1]][c_2]]\cdots[c_m]]}],
\mlabel{eq:snw}
\end{equation}
where $c=c_1c_2\cdots c_m$ with each $c_i\in \alsbw{Z}$ and $c_j \prec_{\rm lex} c_{j+1}$, $1\leq j\leq m-1$. Then
\begin{equation}
[q\suba{u}]_u=q\suba{[u]} + \sum_i \alpha_i q_i\suba{[u]},
\mlabel{eq:quu}
\end{equation}
where each $\alpha_i\in\bfk$, $q_i\in\sopm{Z}$ and $\lbar{[q\suba{u}]_u}=q\suba{u}\ord q_i\suba{u}$.
\mlabel{lem:lalsbw}
\end{lemma}

\begin{proof}
We prove by induction on ${\rm dep}(q\suba{u})\geq 0.$
The initial step of ${\rm dep}(q\suba{u})=0$ follows from~\cite[Lemma 2.4]{QC}.
For the inductive step of ${\rm dep}(q\suba{u})>1$, we have two cases to consider.

\noindent {\bf Case 1.} The $u$ is not contained in the operator $\lc\, \rc$,
that is, $q\suba{u}=v_1\cdots v_k u w_1\cdots w_{\ell}$ for some $v_i,w_j\in\alsbw{Z}$. In this case,
$[q\suba{u}]=[v_1\cdots v_k u w_1\cdots w_{\ell}]$. By~\cite[Lemma 2.4]{QC},
$$[q\suba{u}]={[[v_1]\cdots [v_k] [u w_1\cdots w_{j}] [w_{j+1}]\cdots[w_{\ell}]]}
={[v_1\cdots v_k [u c] w_{j+1}\cdots w_{\ell}]},$$
where $c=w_1\cdots w_{j}$. Write $c=c_1\cdots c_m$ with each $c_i\in \alsbw{Z}$ and $c_i \prec_{\rm lex} c_{i+1}$.
In terms of~\cite[Lemma 2.4]{QC} again,
\begin{align*}
[q\suba{u}]_u=&\ {[[v_1]\cdots [v_k] [[[u] [c_1]]\cdots [c_{m}] [w_{j+1}]]\cdots[w_{\ell}]]},\\
[q\suba{u}]_u=&\ {v_1\cdots v_k [u] w_1\cdots w_{\ell}}+\sum_i \alpha_i q_i\suba{[u]},
\end{align*}
with $q\suba{u}={v_1\cdots v_k u w_1\cdots w_{\ell}} \succ_{\rm lex} q_i\suba{u}$.
By Eq.~(\mref{eq:dlor}), $q\suba{u}={v_1\cdots v_k u w_1\cdots w_{\ell}} \ord q_i\suba{u}$.

\noindent {\bf Case 2.} The $u$ is contained in the operator $\lc\, \rc$, that is, $q\suba{u}=v \blw{p\suba{u}}w$ for some $v,w\in \plie{X}\cup \{\bfone\}$. So $[q\suba{u}]=[v \blw{p\suba{u}}w].$
By the induction hypothesis,
$$[p\suba{u}]=[p'\suba{[uc]}]\,\text{ and }\,[p\suba{u}]_u=p\suba{[u]} + \sum_i \alpha_i p_i\suba{[u]}\,\text{ with }\, p\suba{u}\ord p_i\suba{u}.$$
Consequently,
\begin{align*}
[q\suba{u}]=&\ [v \blw{p\suba{u}}w]=[v \blw{[p'\suba{[uc]}]}w],\\
 [q\suba{u}]_u =&\ [v \blw{[p'\suba{[uc]}]}w]_u = [v \blw{[p'\suba{[uc]}]_u}w] = v \blw{p\suba{[u]} + \sum_i \alpha_i p_i\suba{[u]}}w \\
 =&\ v \blw{p\suba{[u]}}w+\sum_i \alpha_iv \blw{  p_i\suba{[u]}}w =q\suba{[u]}+\sum_i \alpha_i v \blw{  p_i\suba{[u]}}w,\,\text{ with }\, \alpha_i\in\bfk.
\end{align*}
Since $p\suba{u}\ord p_i\suba{u}$, we obtain
$\lbar{[q\suba{u}]_u}=q\suba{u}\ord v \blw{  p_i\suba{[u]}}w$, as required.
\end{proof}

\begin{exam}
Let $Z=\{x,y,z\}$ be a set with $x>y>z$. Let $q=\star z$ and $u=xy$. Then $q\suba{u}=xyz\in\alsbw{Z}$
and so $$[q\suba{u}]=(x(yz))=[xyz]=[\star\suba{uz}]$$
is in $\nlsbw{Z}$.
By Eq.~\meqref{eq:snw},
$$[q\suba{u}]_u=[\star|_{[[xy][z]]}]=[[xy][z]]=[(xy)z]= ((xz)y) + (x(yz))\in\bfk\nlsbw{Z},$$
which can also be expressed in form of Eq.~(\mref{eq:quu}):
$$[q\suba{u}]_u = ((xz)y) + (x(yz)) = ((xy)z)=(xy)z-z(xy)= \star z \suba{[xy]}- z\star\suba{[xy]},$$
with $\star z \suba{xy}  = xyz \ord zxy = z\star\suba{xy}$.
\end{exam}

The above result can be extended from bracketed monomials to bracketed polynomials.

\begin{defn}
Let $q\in\sopm{Z}$ and $f\in \opliez\subseteq \bfk\plie{Z}$ be monic.
We call $$[q\suba{{f}}]_{\lbar{f}}:=[q\suba{\lbar{f}}]_{\lbar{f}}\suba{[\lbar{f}]\mapsto f}$$
a {\bf special normal $f$-word} if $q\suba{\lbar{f}}\in \alsbw{Z}$, where
$[q\suba{\lbar{f}}]_{\lbar{f}}$ is defined by Eq.~\meqref{eq:snw}.
\end{defn}

\begin{corollary}
Let $f\in \opliez\subseteq \bfk\plie{Z}$ be monic and $q\suba{\lbar{f}}\in \alsbw{Z}$. Then
$$[q\suba{f}]_{\lbar{f}}=q\suba{f}+\sum_i \alpha_i q_i\suba{f},$$
where each $\alpha_i\in\bfk$, $q_i\in\sopm{Z}$ and $q\suba{\lbar{f}}\ord q_i\suba{\lbar{f}}$.
\mlabel{lem:jeq}
\end{corollary}

\subsection{\rcp for Lie algebras}
\mlabel{ssec: rpclie}
With the preparation of the previous subsection, we can formulate \rcp for Lie algebras in the contexts of rewriting systems and \gsbs.

\subsubsection{Rewriting systems}
We first recall some basic concepts of rewriting systems from~\mcite{BN,GG,ZGGS}.

Let $V$ be a vector space with a linear basis $W$.
\begin{enumerate}
\item For $f=\sum\limits_{w\in W}c_w w \in V$ with $c_w\in \bfk$, the {\bf  support} $\supp(f)$ of $f$ is the set $\{w\in W\,|\,c_w\neq 0\}$. By convention, we take $\supp(0) = \emptyset$.
\item Let $f, g\in V$. We use $f \dps g$ to indicate the property that $\supp(f) \cap \supp(g) = \emptyset$. If this is the case, we say $f \dps g$ is a {\bf  direct sum} of $f$ and $g$ and use $f\dps g$ to denote the sum $f+ g$.

\item For $f \in V$ and $w \in \supp(f)$ with the coefficient $\alpha_w$, write $R_w(f) := \alpha_w w -f \in V$. So $f = \alpha_w w \dps (-R_w(f))$.
\end{enumerate}

\begin{defn}
Let $V$ be a vector space with a linear basis $W$.
\begin{enumerate}
\item  A {\bf  term-rewriting system $\Pi$ on $V$ with respect to $W$} is a binary relation $\Pi \subseteq W \times V$. An element $(t,v)\in \Pi$ is called a (term-) rewriting rule of $\Pi$, denoted by $t\to v$.

\item The term-rewriting system $\Pi$ is called {\bf simple with respect to $W$} if $t \dps v$ for all $t\to v\in \Pi$.

\item If $f = \alpha_t t\dps (-R_t(f))\in V$, using the rewriting rule $t\to v$, we get a new element $g:= c_t v - R_t(f) \in V$, called a {\bf one-step rewriting} of $f$ and denoted $f \to_\Pi g$ or $\tvarrow{f}{g}{\Pi}$.  \label{item:Trule}

\item The reflexive-transitive closure of $\rightarrow_\Pi$ (as a binary relation on $V$) is denoted by $\astarrow_\Pi$ and, if $f \astarrow_\Pi g$, we say {\bf  $f$ rewrites to $g$ with respect to $\Pi$}. \label{item:rtcl}

\item Two elements  $f, g \in V$ are {\bf  joinable} if there exists $h \in V$ such that $f \astarrow_\Pi h$ and $g \astarrow_\Pi h$; we denote this by $f \downarrow_\Pi g$.

\item A {\bf fork}  is a pair of distinct reduction sequences $(f \astarrow_\Pi g_1, f \astarrow_\Pi g_2)$  starting from the same  $f\in V$. The fork is called {\bf joinable} if $g_1 \downarrow_\Pi g_2$.

\item An element $f\in V$ is {\bf a normal form} if no more rules from $\Pi$ can apply.
\end{enumerate}
\label{def:ARSbasics}
\end{defn}

As usual~\mcite{BN}, a term-rewriting system $\Pi$ on $V$ is called
 \begin{enumerate}
 \item {\bf  terminating} if there is no infinite chain of one-step rewritings
 $$f_0 \rightarrow_\Pi f_1 \rightarrow_\Pi f_2 \rightarrow_\Pi \cdots \quad.$$
\item {\bf  confluent} if every fork  is joinable.
\item {\bf  convergent} if it is both terminating and confluent.
\end{enumerate}

A family of OLPIs give rise to a rewriting system. Let $\Phi \subseteq \oplie\subseteq \bfk\plie{X}$ be monic.
For each $s\in S_\Phi:=\{\phi(u_1,\ldots,u_k)\,|\,\phi\in \Phi,  u_1,\ldots,u_k\in\bfk\frakS(Z)\}$ and $q\in \sopm{Z}$,
if $q\suba{\lbar{s}}\in \alsbw{Z}$, then we may write
$$[q\suba{s}]_{\lbar{s}}:=[q\suba{\lbar{s}}]\dps \big(-R([q\suba{s}]_{\lbar{s}})\big)\in\bfk\nlsbw{Z},$$
which under the order $\ordq$ on $\frakS(Z)$ can be viewed as a rewriting rule
$$[q\suba{\lbar{s}}]\to R([q\suba{s}]_{\lbar{s}}).$$

Let us give an example to illustrate these notions and the induced rewriting rules.

\begin{exam}
Let $\phi(x,y)= [\bws{x}\bws{y}] - \bws{[\bws{x}y]}-\bws{[x\bws{y}]}\in\opliex$. Let $Z$ be a well-ordered set and
$$s=\phi([u], [v]) =[\bws{[u]}\bws{[v]}]-\bws{[\bws{[u]}[v]]}-\bws{[{[u]}\bws{[v]}]}$$
with $u\ord v \in\alsbw{Z}$. For $q_1=\star$, $q_2=\star\bws{w}$ and $ {v}\ord {w}\in\alsbw{Z}$, we have $\lbar{s}=\bws{u}\bws{v}$ and

$$[q_1\suba{s}]_{\lbar{s}}=[q_1\suba{\lbar{s}}]_{\lbar{s}}\suba{[\lbar{s}]\mapsto s}=s=[\bws{[u]}\bws{[v]}] \dps\Big(-\bws{[\bws{[u]}[v]]}-\bws{[{[u]}\bws{[v]}]}\Big).$$
It induces a rewriting rule
$$\big(\bws{[u]}\bws{[v]})\rightarrow \bws{[\bws{[u]}[v]]}+\bws{[{[u]}\bws{[v]}]}\in\nlsbw{Z}\times \opliez.$$
Similarly,
\begin{eqnarray*}
[q_2\suba{s}]_{\lbar{s}}
&=&[q_2\suba{\lbar{s}}]_{\lbar{s}}\suba{[\lbar{s}]\mapsto s}\\
&=&[\star\bws{w}\suba{\bws{u}\bws{v}}]_{\bws{u}\bws{v}}\suba{[{\bws{u}\bws{v}}]\mapsto s}\\
&=&[[\bws{u}\bws{v}][\bws{w}]]\suba{[{\bws{u}\bws{v}}]\mapsto s}\\
&=&\big((\bws{[u]}\bws{[v]})\bws{[w]}\big)\suba{(\bws{[u]}\bws{[v]})\mapsto s}\\
&=&\big(s\bws{[w]}\big)\\
&=&\big([\bws{[u]}\bws{[v]}]\bws{[w]}\big)-   \big(\bws{[\bws{[u]}[v]]}\bws{[w]}\big) - \big(\bws{[{[u]}\bws{[v]}]}\bws{[w]}\big)\\
&=&\big(\bws{[u]}[\bws{[v]}\bws{[w]}]\big)\dps \big([\bws{[u]}\bws{[w]}]\bws{[v]}\big)-   \big(\bws{[\bws{[u]}[v]]}\bws{[w]}\big) - \big(\bws{[{[u]}\bws{[v]}]}\bws{[w]}\big).
\end{eqnarray*}
It induces a rewriting rule
$$\big(\bws{[u]}[\bws{[v]}\bws{[w]}]\big)\rightarrow- \big([\bws{[u]}\bws{[w]}]\bws{[v]}\big)+   \big(\bws{[\bws{[u]}[v]]}\bws{[w]}\big) + \big(\bws{[{[u]}\bws{[v]}]}\bws{[w]}\big)$$
in $\nlsbw{Z}\times \opliez$.
\end{exam}
Next let $Z$ be a well-ordered set. Let $S$ be a monic subset of $\opliez$.   With an invariant monomial order $\ordq$ on $\plie{Z}$, we define a term-rewriting system
\begin{equation}
\Pil{S}:=\bigg\{[q\suba{\lbar{s}}]\rightarrow R([q\suba{s}]_{\lbar{s}})\,\bigg|\,s\in S, q\in \sopm{Z},q\suba{\lbar{s}}\in \alsbw{Z}\bigg\}
\subseteq \nlsbw{Z}\times \opliez.
\mlabel{eq:ltrs}
\end{equation}

\begin{defn}
Let $\Phi\subseteq\oplie\subseteq \bfk\plie{X}$ be a system of monic OLPIs. Let $Z$ be a well-ordered set.
We call $\Phi$ {\bf convergent} if $\Pil{S_\Phi}$ is convergent, where
$$S_\Phi=\{\phi(u_1,\ldots,u_k)\,|\,u_1,\ldots,u_k\in \opliez, \phi\in \Phi\}\subseteq\opliez.$$
\end{defn}

Now we formulate \rcp for Lie algebras in terms of rewriting systems.

\begin{problem}
{\em (\rcp for Lie algebras via rewriting systems)} Determine all convergent systems of OLPIs.
\mlabel{prob:rpcltrs}
\end{problem}

\subsubsection{\gsbs}
We will fix an invariant monomial order $\ordq$ in this subsection. In particular, a \gsb in $\opliez$ is taken with respect to this order. There are two kinds of nontrivial compositions.

\begin{defn}
Let $Z$ be a well-ordered set. Let $f, g \in\opliez\subseteq \bfk\plie{Z}$ be monic.
\begin{enumerate}
\item  If there are $u, v, w\in \plie{Z}$ such that $w = \lbar{f}u = v\lbar{g}$ with
$\max\{ |\lbar{f}|, |\lbar{g}|\}< |w| < |\lbar{f}| + |\lbar{g}|$, then
$$\langle f,g \rangle_w:=\langle f,g \rangle^{u,v}_w:= [fu]_{\lbar{f}} - [vg]_{\lbar{g}}$$
is called the {\bf intersection composition of $f$ and $g$ with respect to $w$}.
\mlabel{item:intcompl}
\item  If there is $q\in\sopm{Z}$ such that $w = \lbar{f} = q\suba{\lbar{ g}},$ then
$$\langle f,g \rangle_w:=\langle f,g \rangle^q_w := f - [q\suba{g}]_{\lbar{g}}$$
is called the {\bf including composition of $f$ and $g$ with respect to $w$}.
\mlabel{item:inccompl}
\end{enumerate}
\mlabel{defn:compl}
\end{defn}
Here comes the main notion on \gsbs.
\begin{defn}
Let $Z$ be a well-ordered set. Let $S\subseteq\opliez\subseteq \bfk\plie{Z}$ be monic.
\begin{enumerate}
\item An element $f\in\opliez$ is called {\bf  trivial modulo $(S, w)$} if
$$f =\sum_i \alpha_i [q_i\suba{s_i}]_{\lbar{s_i}}$$
with $w\ord q_i\suba{\lbar{s_i}}$ for each special normal $s_i$-word $[q_i\suba{s_i}]_{\lbar{s_i}}$, where each $ \alpha_i\in\bfk$, $q_i\in\sopm{Z}$, $s_i\in S$.
In this case, we write $f\equiv 0 \mod(S, w).$
\item We call $S$ a {\bf \gsb} in $\opliez$ with respect to $\ordq$ if, for all pairs $f, g \in S$, every intersection composition of the form $\langle f,g\rangle^{u,v}_w$ is trivial modulo $(S, w)$, and every including composition of the form $\langle f, g\rangle ^q_w$ is trivial modulo $(S, w)$.
\end{enumerate}
\mlabel{def:gsbl}
\end{defn}

Expanding the approach in~\mcite{BCQ}, we establish the Composition-Diamond lemma for operated Lie algebras with respect to an invariant monomial order $\ordq$ on $\plie{Z}$.

\begin{lemma}
Let $\ordq$ be an invariant monomial order  on $\plie{Z}$ and $f,g\in \opliez\subseteq\bfk \plie{Z}$ monic.
Then
$$\langle f, g\rangle_w-(f, g)_w\equiv_{\rm ass} 0 \mod(\{f, g\}, w),$$
where $(f, g)_w$ and $\equiv_{\rm ass}0$ are  the composition of $f$ and $g$, and the triviality modulo $(\{f, g\},w)$ in associative framework.
\mlabel{lem:aslie}
\end{lemma}
\begin{proof}
There are two cases to consider.

\noindent {\bf Case 1.} $w = \lbar{f}u = v\lbar{g}$ for some $u, v\in \plie{Z}$. Then
$$\langle f, g\rangle_w=[fu]_{\lbar{f}} - [vg]_{\lbar{g}}\,\text{ and }\,
(f, g)_w=fu-vg.$$
By Corollary~\mref{lem:jeq},
$$[fu]_{\lbar{f}} - [vg]_{\lbar{g}}=fu+\sum_i \alpha_i p_i\suba{f}-vg-\sum_j \beta_j q_j\suba{g}$$
 for some $\alpha_i, \beta_j\in\bfk$
and $p_i, q_j\in\sopm{Z}$ with $w\ord p_i\suba{\lbar{f}}$ and $w\ord q_j\suba{\lbar{g}}$.
Therefore
$$\langle f, g\rangle_w-(f, g)_w=\sum_i \alpha_i p_i\suba{f}-\sum_j \beta_j q_j\suba{g}\equiv_{\rm ass} 0 \mod(\{f, g\}, w).$$

\noindent {\bf Case 2.} $w = \lbar{f} = q\suba{\lbar{ g}}$  for $q\in\sopm{Z}$. Then
$$\langle f,g \rangle_w= f - [q\suba{g}]_{\lbar{g}}=f-q\suba{g}-\sum_i \alpha_i q_i\suba{g}$$
for some $\alpha_i\in\bfk$ and $q_i\in\sopm{Z}$ with $q\suba{\lbar{g}}\ord q_i\suba{\lbar{g}}$.
Further by $(f, g)_w=f-q\suba{g}$,
$$\langle f, g\rangle_w-(f, g)_w=-\sum_i \alpha_i q_i\suba{g}\equiv_{\rm ass} 0 \mod(\{f, g\}, w).$$
This completes the proof.
\end{proof}

\begin{lemma}
Let $\ordq$ be an invariant monomial order  on $\plie{Z}$ and $S\subseteq\opliez\subseteq\bfk \plie{Z}$ monic. Then $S$ is a \gsb in $\opliez$ if and only if $S$ is a \gsb in $\bfk\plie{Z}$.
 \mlabel{lem:gseq}
\end{lemma}
\begin{proof}
By Definition~\mref{defn:compl} and Lemma~\mref{lem:ldt}, for any $f,g \in S$, they have compositions in $\opliez$ if and only if so do in $\bfk\plie{Z}$.

Assume that $S$ is a \gsb in $\opliez$.
By Definition~\mref{def:gsbl},
for any composition $\langle f, g\rangle_w$ from $S$, we obtain
$$\langle f, g\rangle_w=\sum_i \alpha_i [q_i\suba{s_i}]_{\lbar{s_i}}\,\text{ for }\,
\alpha_i\in\bfk, s_i\in S, q_i\in\sopm{Z}$$
with $w\ord q_i\suba{\lbar{s_i}}.$
Further by Corollary~\mref{lem:jeq},
\begin{align*}
\langle f, g\rangle_w=&\ \sum_i \alpha_i [q_i\suba{s_i}]_{\lbar{s_i}}
=\sum_i \alpha_i \Big(q_{i}\suba{s_i}+\sum_j \beta_{ij} q_{ij}\suba{s_{ij}}\Big)
\,\text{ for }\,\beta_{ij}\in\bfk, s_{ij}\in S, q_{ij}\in\sopm{Z},\\
[q\suba{f}]_{\lbar{f}}=&\ q\suba{f}+\sum_i \alpha_i q_i\suba{f},\,\text{ where }\,
\alpha_i\in\bfk, q_i\in\sopm{Z}, q\suba{\lbar{f}}\ord q_i\suba{\lbar{f}}.
\end{align*}
Using Lemma~\mref{lem:aslie}, we have
$$\langle f, g\rangle_w-(f, g)_w=\sum_i \alpha_i \Big(q_{i}\suba{s_i}+\sum_j \beta_{ij} q_{ij}\suba{s_{ij}}\Big)-(f, g)_w\equiv_{\rm ass} 0 \mod(S, w).$$
Hence $(f, g)_w\equiv_{\rm ass} 0 \mod(S, w)$ and so $S$ is a \gsb in $\bfk\plie{Z}$.

Conversely, supposed that $S$ is a \gsb in $\bfk\plie{Z}$.
For any composition $( f, g)_w$ from $S$, we have $(f, g)_w\equiv_{\rm ass} 0 \mod(S, w)$.
Again by Lemma~\mref{lem:aslie},
$$\langle f, g\rangle_w\equiv_{\rm ass} 0 \mod(S, w),$$
and so $\langle f, g\rangle_w\in \Id(S)$. Further by Composition-Diamond lemma for operated algebras in~\mcite{BCQ}, write
$$\langle f, g\rangle_w=\sum_{i=1}^n \alpha_iq_i\suba{s_i}\,\text{ for }\,\alpha_i\in\bfk, s_i\in S, q_i\in\sopm{Z}$$
with $w\ord q_1\suba{\lbar{s_i}}\ord\cdots\ord q_n\suba{\lbar{s_n}}.$
Since $\langle f, g\rangle_w\in\opliez,$
$\lbar{\langle f, g\rangle_w}=q_1\suba{\lbar{s_i}}\in \alsbw{Z}.$
Let
$$\langle f, g\rangle_w^0:=\langle f, g\rangle_w\,\text{ and }\,\langle f, g\rangle_w^1:=\langle f, g\rangle_w-\alpha_1[q_i\suba{s_i}]_{\lbar{s_1}}\in\opliez.$$
Then $\lbar{\langle f, g\rangle_w}=q_1\suba{\lbar{s_i}}\ord\lbar{\langle f, g\rangle_w^1}$.
By Corollary~\mref{lem:jeq} and $\langle f, g\rangle_w\equiv_{\rm ass} 0 \mod(S, w),$ we have
$$\langle f, g\rangle_w^1\equiv_{\rm ass} 0 \mod(S, w).$$
Using induction on $\langle f, g\rangle_w^n$,  we have
$$\langle f, g\rangle_w^n=\langle f, g\rangle_w^{n-1}-\beta_{n-1}[p_{n-1}\suba{s_{n-1}}]_{\lbar{{s_{n-1}}}}.$$
Therefore
$$\langle f, g\rangle_w=\sum_i\beta_{i}[p_{i}\suba{s_{i}}]_{\lbar{{s_{i}}}},$$
with $w\ord p_{i}\suba{\lbar{s_{i}}}.$ This proves that $S$ is a \gsb in $\opliez$.
\end{proof}

Denote by $\Idl(S)$  the Lie ideal of $\opliez$ generated by $S\subseteq\opliez.$ Then
\begin{equation}
\begin{split}
\Idl(S)=~\bigg\{\sum_{i=1}^n \alpha_i[{q_i} \suba{s_i}]_{\lbar{s_i}}\,\bigg|\, n\geq1, \alpha_i\in\bfk, s_i\in S\, \text{ and }\, q_i\in\sopm{Z} \bigg\}.
\end{split}
\mlabel{eq:ideal}
\end{equation}
Define
$$\irrl(S):=\big\{[w]\,\big|\, w\in\alsbw{Z}, w\neq q\suba{\lbar{s}}\,\text{ for }\, s\in S\,\text{ and }\, q\in\sopm{Z}\big\}.$$

\begin{lemma}
Let $\ordq$ be an invariant monomial order  on $\plie{Z}$ and $S\subseteq\opliez\subseteq \bfk\plie{Z}$ monic. Then, for any $f\in\opliez$, $f$ has a representation:
$$f=\sum_i \alpha_i[u_i]+\sum_j \beta_j [q_i\suba{s_i}]_{\lbar{s_i}}$$
where each $\alpha_i, \beta_j\in\bfk$, $s_i\in S$, $[u_i]\in \irrl(S)$, $\lbar{f}\ordq u_i$ and $\lbar{f}\ordq q_i\suba{\lbar{s_i}}$.
\mlabel{lem:rep}
\end{lemma}
\begin{proof}
Since $f\in\opliez$, we can write $f=\sum_i \alpha_i[u_i]$,
where each $[u_i]\in\nlsbw{Z}$, $\alpha_i\in\bfk\setminus\{0\}$ and
$u_1\ord u_2\ord \cdots$. If $[u_1]\in\irrl(S)$, then let
$f_1:=f-\alpha_1[u_1]$. If $[u_1]\notin\irrl(S)$, then there exists $s\in S$ and $q\in\sopm{Z}$ such that
$u_1=q\suba{s}$. In this case, let
$$f_1:=f-[q\suba{s}]_{\lbar{s}}.$$
Hence, in both cases, we have $\lbar{f}\ord \lbar{f_1}.$
Then the result follows from the induction on $\lbar{f}$.
\end{proof}

Now, we prove that the following Composition-Diamond Lemma for \olies.

\begin{lemma}{\em (Composition-Diamond lemma for \olies)} Let $Z$ be a well-ordered set and $\ordq$ an invariant monomial order on $\plie{Z}$. Let $S\subseteq\opliez\subseteq \bfk\plie{Z}$ be monic. The following conditions are equivalent.
\begin{enumerate}
\item $S$ is a  \gsb in $\opliez$. \mlabel{item:cda}
\item  For all $f \neq 0$ in $\Idl ( S )$ , $\lbar{f} = q\suba{\lbar{s}}\in\alsbw{Z}$ for some $q\in\sopm{Z}$ and $s \in $S. \mlabel{item:cdb}
\item  $\opliez=\bfk \irrl( S )\oplus\Idl(S)$ and $\irrl( S )$ is a \bfk-basis of $\opliez/ \Idl(S)$. \mlabel{item:cdc}
\end{enumerate}
\mlabel{lem:cdmol}
\end{lemma}
\begin{proof}
\meqref{item:cda}$\Rightarrow$ \meqref{item:cdb}
Notice that $f\in \Idl( S )\subseteq \Id( S )$, where  $\Id( S )$ is the ideal of $\bfk\plie{Z}$ generated by $S$. By Lemma~\mref{lem:gseq} and the Composition-Diamond lemma for operated algebras in~\mcite{BCQ}, we have $\lbar{f}=q\suba{\lbar{s}}$ for some $s\in S$ and $q\in\sopm{Z}$.

\meqref{item:cdb}$\Rightarrow$ \meqref{item:cdc}
First, we show that elements in $\irrl( S )$ are linear independent. Assume that $\sum \alpha_i [u_i]=0$ in $\opliez/ \Idl(S)$ with $u_1\ord u_2 \ord \cdots$, for each $[u_i]\in \irrl( S )$. Then
$\sum \alpha_i [u_i]\in \Idl(S).$
If there exists the first non-zero coefficient $\alpha_k$, then $\lbar{\sum \alpha_i [u_i]}=u_k$. Further by~\meqref{item:cdb}, we have $u_k=q\suba{s}\in\alsbw{Z}$ for some $s\in S$ and so
$[u_k]\notin\irrl( S )$, a contradiction. Thus $\alpha_i=0$ for all $i$. Second, by Lemma~\mref{lem:rep}, for any $f\in\opliez$, we have
$$f=\sum_i \alpha_i[u_i]+\sum_j \beta_j [q_i\suba{s_i}]_{\lbar{s_i}},\,[u_i]\in\irrl( S ).$$
Therefore,
$$f+\Idl(S)=\sum_i \alpha_i[u_i]+\sum_j \beta_j [q_i\suba{s_i}]_{\lbar{s_i}}+\Idl(S)=\sum_i \alpha_i[u_i]+\Idl(S),\,[u_i]\in\irrl( S ),$$
and so  $\irrl( S )$ is a linear basis of $\opliez/ \Idl(S)$.

\meqref{item:cdc}$\Rightarrow$ \meqref{item:cda}
 For any composition $\langle f, g\rangle_w$ of $f,g \in S$, we know that $\langle f, g\rangle_w\in \Idl(S).$
 By Lemma~\mref{lem:rep},
 $$\langle f, g\rangle_w = \sum_j \beta_j [q_i\suba{s_i}]_{\lbar{s_i}},$$
 where each $\beta_j\in\bfk$, $s_i\in S$ and $w\ord \lbar{\langle f, g\rangle_w}\ordq q_i\suba{\lbar{s_i}}$.
 Therefore  $S$ is a  \gsb in $\opliez$.
\end{proof}

Then \rcp can be formulated for Lie algebras as follows in terms of \gsbs.

\begin{defn}
Let $\Phi\subseteq\oplie\subseteq \bfk\plie{X}$ be a system of monic OLPIs. Let $Z$ be a well-ordered set.
We call $\Phi$ a {\bf \GS system} if ${S_\Phi}$ is a \gsb in $\opliez$, where
$$S_\Phi=\{\phi(u_1,\ldots,u_k)\,|\,u_1,\ldots,u_k\in \opliez, \phi\in \Phi\}\subseteq\opliez.$$
\end{defn}

\begin{problem}
{\em (\rcp for Lie algebras via \gsbs)} Determine all \GS systems of OLPIs.
\mlabel{prob:rpclgsb}
\end{problem}

\subsection{Equivalence of the two formulations}
In this subsection, we will establish a connection between the two formulations of \rcp for Lie algebras, stated in  Problem~\mref{prob:rpcltrs} and Problem~\mref{prob:rpclgsb} respectively.

An element $f\in\opliez$ is called {\bf $S$-reducible} if there exists $g\in\opliez$, $g \neq f$, such that $f \astarrow_{\Pil{S}} g$.
The following result gives a sufficient condition for terminating.

\begin{lemma}
Let $S$ be a monic set of  $\opliez$, and let $\Pil{S}$ be the term-rewriting system from $S$ in Eq.~\meqref{eq:ltrs} with an invariant monomial order $\ordq$ on $\plie{Z}$. Then
$\Pil{S}$ is terminating.
\mlabel{lem:termi}
\end{lemma}
\begin{proof}
Let
$$\mathcal{C}=\big\{f\in\opliez \,\big|\,\text{ there is an infinite $S$-reduction chain } f:=f_0 \to_{\Pil{S}} f_1\to_{\Pil{S}}\cdots \}.$$
We only need to prove that $\mathcal{C}=\emptyset$. Suppose to the contrary that $\mathcal{C}\neq\emptyset$.  Since $f$ is $S$-reducible for all $f\in\mathcal{C}$ and $\ordq$ is a well-order on $\plie{Z}$, the set
$$\mathcal{L}:=\{L(f):={\rm max}\{w\in\plie{Z}\,|\, [w] \text{ is a monomial of $f$ and $w$ is $S$-reducible }\}\}$$
is non-empty and has a least element $w_0$. We fix an $f\in\mathcal{C}$  with $L(f)=w_0$ and
and fix an infinite $S$-reduction chain
$$f:=f_0 \to_{\Pil{S}} f_1\to_{\Pil{S}}\cdots.$$
Then we have $f_i\in\mathcal{C}$ and hence $S$-reducible for all $i \geq 1$.
Let $w_i=L(f_i)$. By the definition of $\Pil{S}$, $w_0\geq w_i\geq\cdots.$
Since every $f_i$ is in $\mathcal{C}$, and $w_0$ is the least element in $\mathcal{L}$, we must have $w_0=w_i$ for all $i\geq1.$
Let $g_i = f_i -\alpha_i [w_i]$, where $\alpha_i$ is the coefficient of $w_i$ in $f_i$. Then we have the infinite reduction sequence
$g_0 \to_{\Pil{S}} g_1\to_{\Pil{S}}\cdots.$
and $L(g_0) < L(f)=L(f_0)$, by none of the $[w_i]$ is involved in the $S$-reduction of the fixed sequence starting with $f$. This is a contradiction, showing that $\mathcal{C}=\emptyset$. This completes the proof.
\end{proof}

Thanks to Lemma~\mref{lem:termi}, we obtain

\begin{lemma}
Let $Z$ be a well-ordered set. Let $S$ be a monic set of  $\opliez$, and let $\Pil{S}$ be the term-rewriting system from $S$ in Eq.~\meqref{eq:ltrs}.
With an invariant monomial order $\ordq$ on $\plie{Z}$.
\begin{enumerate}
\item If $f\astarrow_{\Pil{S}} g$ for $f, g \in\opliez$, then $f - g \in \Idl(S)$.\mlabel{item:rta}
\item $\opliez=\bfk\irrl(S)+\Idl(S).$ \mlabel{item:rte}
\item If ${\Pil{S}}$ is confluent, then $f\in\Idl(S)$ if and only if $f \astarrow_{\Pil{S}} 0$.\mlabel{item:rtb}
\item ${\Pil{S}}$ is confluent if and only if $\bfk\irrl(S)\cap\Idl(S)=0$.\mlabel{item:rtc}
\end{enumerate}
\mlabel{lem:rtsc}
\end{lemma}
\begin{proof}
\meqref{item:rta} If $f=g$, then $f-g\in \Id_{\rm Lie}(S)$. Suppose $f\neq g$ and $n \geq 1$ is the minimum number such that $f$ rewrites to $g$ in $n$ steps. We prove the result by induction on $n$.
For $n=1$, we write
$$f=\alpha [q\suba{\lbar{s}}]\dps h \to_{\Pil{S}} \alpha[R([q\suba{s}]_{\lbar{s}})]+h=g,$$
with $\alpha\in\bfk$, $s\in S$ and $h\in\opliez$. By 
Eq.~\meqref{eq:ltrs},
$$f-g=\alpha[q\suba{\lbar{s}}]-\alpha[R([q\suba{s}]_{\lbar{s}})]=\alpha[q\suba{s}]_{\lbar{s}}\in \Idl(S).$$
Assume that the result is true for $n \geq 1$ and consider the case of $n+1$.
Then we have the following rewritings
$$f\to_{\Pil{S}} f' \astarrow_{\Pil{S}} g$$
for some $f \neq f' \in\opliez$ and $f'$ rewrites to $g$ in $n$ steps. By the induction hypothesis,
$f - f' \in\Idl(S)$ and $f' - g \in\Idl(S)$. Thus $f - g \in\Idl(S)$, as required.

\meqref{item:rte} For any $f\in \opliez$, by Lemma~\mref{lem:termi}, ${\Pil{S}}$ is terminating. Thus there is $g\in\bfk\irrl(S)$ such that
$f$ has a normal form $g$. By Item~\meqref{item:rta}, $f-g\in\Idl(S)$ and so $f \in \irrl(S) + \Idl(S)$.

\meqref{item:rtb} For $f\in\Idl(S)$, by Eq.~\meqref{eq:ideal}, we have
$$f=\sum_{i=1}^n \alpha_i[{q_i} \suba{s_i}]_{\lbar{s_i}},\, \text{ where each }\, \alpha_i\in\bfk, s_i\in S\,\text{ and }\, q_i\in\sopm{Z}.$$
For each $[q\suba{s_i}]_{\lbar{s_i}}$ with  $1\leq i \leq n$, we may write
$$[q\suba{s_i}]_{\lbar{s_i}}=[q\suba{\lbar{s_i}}]\dps \big(-[R([q\suba{s_i}]_{\lbar{s_i}})]\big).$$
Thus
$$\alpha_i[{q_i} \suba{s_i}]_{\lbar{s_i}}=\alpha_i[q\suba{\lbar{s_i}}]-\alpha_i[R([q\suba{s_i}]_{\lbar{s_i}})]\to_{\Pil{S}} \alpha_i[R([q\suba{s_i}]_{\lbar{s_i}})]-\alpha_i[R([q\suba{s_i}]_{\lbar{s_i}})]=0 ,$$
and so $f \astarrow_{\Pil{S}} 0$. Conversely, if $f \astarrow_{\Pil{S}} 0$, then $f$ is in $\Idl(S)$ from Item~\meqref{item:rtb}.

\meqref{item:rtc}
Suppose that ${\Pil{S}}$ is not confluent. By Lemma~\mref{lem:termi}, ${\Pil{S}}$ is terminating. So there is an
$f\in \opliez$ such that $f$ has two distinct normal forms, say $g$ and $h$. Then $g,h\in\bfk\irrl(S)$
and so $g-h\in\bfk\irrl(S)$. By Item~\meqref{item:rta}, $f-g\in\Idl(S)$ and $f-h\in\Idl(S)$. Hence
$0\neq g-h\in\irrl(S) \cap\Idl(S)$, this is a contradiction. Conversely, suppose $\bfk\irrl(S)\cap\Idl(S)\neq0$. Let $0\neq f\in\bfk\irrl(S)\cap\Idl(S)$. Since $ f\in\bfk\irrl(S)$, it is in normal form. On the other hand, from $f\in\Idl(S)$ and Item~\meqref{item:rtb},
we have $f \astarrow_{\Pil{S}} 0$. Thus $w$ has two normal forms $f$ and $0$, contradicting that ${\Pil{S}}$ is confluent.
\end{proof}
Now we can give several equivalent conditions for \gsbs.

\begin{theorem}
Let $Z$ be a well-ordered set. Let $S$ be a monic set of  $\opliez$, and let $\Pil{S}$ be the term-rewriting system from $S$ in Eq.~\meqref{eq:ltrs}. With an invariant monomial order $\ordq$ on $\plie{Z}$,
the following statements are equivalent.
\begin{enumerate}
  \item $\Pil{S}$ is convergent. \mlabel{item,gsbrtsa}
  \item $\Pil{S}$ is confluent. \mlabel{item,gsbrtsb}
  \item $\Idl(S) \oplus \bfk\irrl(S) =\opliez$. \mlabel{item,gsbrtsd}
  \item $S$ is a \gsb in $\opliez$ with respect to $\ordq$.  \mlabel{item,gsbrtse}
\end{enumerate}
\mlabel{lem:liegreq}
\end{theorem}
\begin{proof}
By Lemma~\mref{lem:termi}, $\Pil{S}$ is terminating.
Then Item~\meqref{item,gsbrtsa} and Item~\meqref{item,gsbrtsb} are equivalent.
Further by Items~\meqref{item:rte} and~\meqref{item:rtc} in Lemma~\mref{lem:rtsc}, Item~\meqref{item,gsbrtsb} and Item~\meqref{item,gsbrtsd} are equivalent.
The equivalence of Item~\meqref{item,gsbrtsd} and Item~\meqref{item,gsbrtse} is obtained by Composition-Diamond Lemma~\mref{lem:cdmol}.
\end{proof}

Now we are ready to give the relationship between the two formulations of \rcp for Lie algebras.
\begin{corollary}
With an invariant monomial order $\ordq$, the two formulations of \rcp for Lie algebras in Problem~\mref{prob:rpcltrs} and
Problem~\mref{prob:rpclgsb} are equivalent.
\mlabel{coro:rcpleq}
\end{corollary}

\subsection{Relationship between the program for associative algebras and for Lie algebras}
\mlabel{ssec:asslie}
We now establish the relationship between \rcp for associative algebras and the one for Lie algebras.

\subsubsection{\rcp for associative algebras}
We first recall \rcp for associative algebras~\mcite{GG,GSZ,ZGGS}.

Let $X$ be a set. Similar to the definition of $\plie{X}$ in Section~\mref{ssec:aslsw}, we define the free operated monoid $\opmxm$ by replacing $S$ with $M$.
Also analogous to $\sopm{X}$, we can define $\sopma{X}$. See~\mcite{GG} for more details. Let $\phi\subseteq\opmx$. We call $\phi=0$ (or simply $\phi$) an {\bf operated polynomial identity (OPI)}.

We now recall some basic concepts for operated algebras.
\begin{defn}
Let  $X$ be a set.
\begin{enumerate}
  \item For an operated algebra $(R,P)$, we say that
  $R$ is a {\bf $\phi$-algebra} and $P$ is a {\bf $\phi$-operator}, if
  $$\phi(r_1, \ldots ,r_k)=0,\quad \text{for all } r_1, \ldots ,r_k\in R.$$
  More generally, for a subset $\Phi\subseteq\opmx$, we call $R$ (resp. $P$) a {\bf $\Phi$-algebra} (resp. {\bf $\Phi$-operator}) if $R$ (resp. $P$)
is a $\phi$-algebra (resp. $\phi$-operator) for each $\phi\in\Phi$.

  \item Let $R$ be an algebra. A linear operator $P$ on $R$ is called {\bf algebraic} if there is $0\neq \Phi\subseteq \opmx$ such that $P$ is a $\Phi$-operator.
\end{enumerate}
\mlabel{de:algop}
\end{defn}

\begin{defn}
Let $\geq$ be a monomial order on $\opmzm$ and $S$ be a monic subset of $\opmz$. We define
\begin{equation}
\Pia{S}:=\{q\suba{\lbar{s}}\rightarrow q\suba{R(s)}\,|\, q \in \sopma{Z}, s=\lbar{s}-R(s)\in S\}\subseteq\opmzm\times \opmz.
\mlabel{eq:trs}
\end{equation}
\end{defn}

\begin{defn}
Let $X$ be a set, and let $\Phi\subseteq\opmx$ be a system of OPIs.
We call $\Phi$ {\bf convergent} if $\Pia{\Phi}$ is convergent.
\end{defn}

On the one hand, \rcp can be interpreted in terms of rewriting systems.

\begin{problem}\mcite{GG}
{\em (\rcp via rewriting systems)}. Determine all convergent systems of OPIs.
\mlabel{prob:rpctrs}
\end{problem}

On the other hand, \rcp can also be interpreted in terms of \gsbs.

\begin{problem}\mcite{GG}
{\em (\rcp via \gsbs)}. Determine all \gsbs systems of OPIs.
\mlabel{prob:rpcgsb}
\end{problem}

The relationship between Problem~\mref{prob:rpctrs} and Problem~\mref{prob:rpcgsb} has also been studied.
For this, we need the relationship between a \gsb of OPIs and a convergent rewriting system of OPIs.  We denote by $\Id(S)$ the ideal of $\opmz$ generated by $S$.
Denote
$$\irr(S):=\left\{w\,\left |\,w\in\opmzm,  w\neq q\suba{\lbar{s}}\, \text{ for }\, s\in S \,\text{ and }\, q\in\sopma{Z}\right . \right\}.$$

\begin{lemma}\mcite{GG}
Let $Z$ be a set, and let $\geq$ be a monomial order on $\opmzm$. Let $S$
be a monic subset of $\opmz$, and let $\Pia{S}$ be the term-rewriting system from $S$ in Eq.~\meqref{eq:trs}. Then
the following statements are equivalent.
\begin{enumerate}
  \item $\Pia{S}$ is convergent.
  \item $\Pia{S}$ is confluent.
  \item $\Id(S) \cap \bfk\irr(S) = 0$.
  \item $\Id(S) \oplus \bfk\irr(S) =\opmz$.
  \item $S$ is a \gsb in $\opmz$ with respect to $\geq$.
\end{enumerate}
\mlabel{lem:asgreq}
\end{lemma}

\begin{remark}
If the set $S$ does not contain $\bfone$, then we obtain a nonunital version of Lemma~\mref{lem:asgreq}.
This will be used later.
\mlabel{rem:none}
\end{remark}

\begin{corollary}\mcite{GG}
With a monomial order on $\opmzm$, the two versions Problem~\mref{prob:rpctrs} and
Problem~\mref{prob:rpcgsb} of \rcp are equivalent.
\mlabel{coro:rcpeq}
\end{corollary}

\subsubsection{Relationship between \rcp between two types of algebras}
Chen and Qiu~\mcite{QC} studied the relationship between \gsbs in free operated associative algebras and \gsbs in free operated algebras with respect to the order $\ordqc$.
In fact this relationship also works with respect to an invariant monomial order $\ordq$.

For any element $[w]$ in $\opliez\subseteq\bfk \plie{Z}$,  there is a unique $w$ in $\bfk\plie{Z}$.
Based on this, we have the following equivalence.

\begin{theorem}
Let $S\subseteq\opliez\subseteq \bfk\frakS(Z)$ be monic. With an invariant monomial order $\ordq$ on $\plie{Z}$, the following statements are equivalent.
\begin{enumerate}
\item $\Pil{S}$ is convergent. \mlabel{item:aleqa}
\item $S$ is a \gsb in $\opliez$.\mlabel{item:aleqb}
\item $\Pia{S}$ is convergent.\mlabel{item:aleqc}
\item $S$ is a \gsb in $\bfk\frakS(Z)$.\mlabel{item:aleqd}
\end{enumerate}
In other words,
under the above hypothesis, the Problems~\ref{prob:rpcltrs}, ~\ref{prob:rpclgsb}, ~\ref{prob:rpctrs} and~\ref{prob:rpcgsb} are equivalent to each other with respect to the order $\ordq$.
\mlabel{thm:alrcpeq}
\end{theorem}
\begin{proof}
By Theorem~\mref{lem:liegreq}, Lemma~\mref{lem:asgreq} and Remark~\mref{rem:none},
Item~\meqref{item:aleqa} is equivalent to Item~\meqref{item:aleqb} and Item~\meqref{item:aleqc} is equivalent to Item~\meqref{item:aleqd}.
Further by Lemma~\mref{lem:gseq} and the order $\ordq$,
Item~\meqref{item:aleqb} is equivalent to Item~\meqref{item:aleqd}.
\end{proof}

\begin{remark}
Let $S\subseteq\opliez\subseteq\bfk \plie{Z}$ be monic OLPIs.
If $S\subseteq\bfk \plie{Z}$ give a ``good" operator on an associative algebra via the classification in Problem~\mref{prob:rpctrs} or Problem~\mref{prob:rpcgsb}, then
$S\subseteq\opliez$ also give a ``good" operator on a Lie algebra via the classification in Problem~\mref{prob:rpcltrs} or Problem~\mref{prob:rpclgsb}. See next section.
\mlabel{coro:rcpasl}
\end{remark}

\section{Appliciations}
\mlabel{sec:app}
In this section, we apply Theorem~\mref{thm:alrcpeq} to find some OPLIs that are \gsbs in free \olies.
They are related to two particular classes of operators, namely, those that generalize differential or Rota-Baxter operators respectively, were studied in~\mcite{GSZ, ZGGS}.
The first one is the class of differential type operators and the second one is the class of Rota-Baxter type operators.

\subsection{Differential type OPIs and OLPIs}
We first recall the notion of differential type OPIs. 
\begin{defn}\mcite{GSZ}
A {\bf differential type OPI}, defining a differential type operator, is
$$\phi(x , y ) := \bws{xy}- N(x , y ),$$
where
\begin{enumerate}
  \item $N(x, y )$ is multi-linear in $x$ and $y$ ;
  \mlabel{item:a}
  \item $N(x, y )$ is a normal $\phi$-form, that is, $N(x , y )$  does not contain any subword of the form $\bws{uv}$, for any non-units $u,v\in \opmx$;
  \mlabel{item:b}
  \item For any set $Z$ with $u, v, w \in \opmzm\backslash\{\bfone\}$,
$N(uv,w) - N(u,vw)\astarrow_{\Pia{\phi}} 0.$
  \mlabel{item:c}
\end{enumerate}
\mlabel{defn:difftyp}
\end{defn}

\begin{remark}
Condition~\meqref{item:a} is imposed so that the operators are linear. Condition~\meqref{item:b}
is needed to avoid infinite rewriting under $\Pi_\phi$. Condition~\meqref{item:c} is needed so that the rewriting $\lc uv\rc \to N(u,v)$ is compatible with the associativity
$(uv)w=u(vw)$.
\end{remark}

A classification of differential type OPIs was proposed in~\mcite{GSZ}.

\begin{theorem} {\em (Classification of differential type operators)}~\mcite{GSZ}
Let $a , b , c , e\in\bfk$. The OPI $\phi(x,y) := \bws{xy} - N(x ,y)$, where $N( x , y )$  is taken from the list below, is of differential type.
\begin{enumerate}
  \item $b ( x \bws{ y }+\bws{ x } y )+ c \bws{ x }\bws{ y }+ exy $ where $b^2 = b + ce$,
  \item  $ce^2 yx + exy + c \bws{ y }\bws{ x }- ce ( y \bws{ x }+\bws{ y } x )$,
  \item $\sum_{i,j\geq0}a_{ij} \bws{ \bfone }^i xy \bws{ \bfone }^j$ with the convention that $\bws{ \bfone }^0= \bfone$,
  \item $x \bws{ y }+\bws{ x } y + ax \bws{ \bfone } y + bxy,$
  \item $\bws{ x } y + a ( x \bws{ \bfone } y - xy \bws{ \bfone }),$
  \item $ x \bws{ y }+ a ( x \bws{ \bfone } y -\bws{ \bfone } xy ) $.
\end{enumerate}
\mlabel{thm:cdto}
\end{theorem}

It was conjectured that this list includes all differential type operators.

Parallel to Definition~\mref{defn:difftyp},  we propose the following concept in the context of Lie algebras.

\begin{defn}
A {\bf differential type OLPI} is
$$\phi(x , y ) := \bws{[x y]}- [N(x , y )]\in\opliex, $$
 with $x> y$ such that
\begin{enumerate}
\item $[N(x , y )]$ is multi-linear in $x$ and $y$;
\mlabel{item:adi}
\item $[N(x , y )]$ is a normal $\phi$-form, that is, $[N(x , y )]$  does not
contain subwords $\bws{[[u][v]]}$ for any $u,v\in \alsbw{\{x,y\}}$ with $u\ord v$;
\mlabel{item:bdi}
\item
For any well-ordered set $Z$ with $u\ord v\ord w \in \alsbw{Z}$,
$$[N([[u][w]],[v])]-
[N([[u][v]],[w])]+ [N([u],[[v][w]])]\astarrow_{\Pil{S_\phi}} 0,$$
where
\begin{equation*}
 \Pil{S_\phi}:=\left\{[q\suba{\bws{u  v}}]\rightarrow R([q\suba{\phi(u , v )}]_{{\bws{u  v}}})\,|\,q\in\sopm{Z}, u, v\in \alsbw{Z}, u\ord v \right\}.
\end{equation*}
\mlabel{item:cdi}
\end{enumerate}
A linear operator on a Lie algebra satisfying a differential type OLPI is called a {\bf differential type operator} on a Lie algebra.
\mlabel{defn:ldifftyp}
\end{defn}

\begin{remark}
Condition~\meqref{item:cdi} is needed so that the rewriting
$\lc [uv]\rc \to [N(u,v)]$ is compatible with the Jacobi identity.
\end{remark}

Now we propose

\begin{problem}{\em (\rcp for Lie algebras: the differential case)} Find all OLPIs of differential
type by finding all expressions $[N ( x , y )]\in \oplie $ of differential type.
\mlabel{pro:dp}
\end{problem}

We now characterize differential type operators in terms of rewriting systems and \gsbs in the free \olies and the free operated associative algebras.

Let $Z$ be a well-ordered set. Define $\der{z}{n}\in\plie{Z}$, $n \geq 0$, recursively by
$$\der{z}{0}:=0,\, \der{z}{n+1}:= \bws{\der{z}{n}}\,\text{ for }\, m\geq0.$$
Denote
\begin{equation}
\Delta(Z):=\{\der{z}{n}\,|\,z\in Z, n\geq0\}.\mlabel{eq:derf}
\end{equation}

\begin{theorem}
Let $Z$ be a well-ordered set. Let
$$\phi(x , y ) = \bws{[x y]}- [N(x,y )] \in\opliex \subseteq \bfk\plie{\{x,y\}},$$
 with $x> y$ and $[N(x,y )]$ satisfies the conditions ~\meqref{item:adi} and~\meqref{item:bdi} in Definition~\mref{defn:ldifftyp}.
Let $\ordqd$ be an invariant monomial order satisfying $\lbar{\phi(u , v)}= \bws{u v}$ for $u\ordd v\in \alsbwo{Z}{\ordqd}$ $($such as $\geq_{\rm dt}$ given in Example~\mref{ex:inva2}$)$. The following statements are equivalent.
\begin{enumerate}
\item $\phi(x, y)= \bws{[x y]}- [N(x,y )](\in \bfk\opliex)$ is a differential type OLPI. \mlabel{item:drga}
\item The term-rewriting system induced in Eq.~\meqref{eq:ltrs}$:$
  $$\Pil{S_\phi}=\left\{\left .[q\suba{\bws{u  v}}]\rightarrow R([q\suba{\phi(u , v )}]_{{\bws{u  v}}})\,\right |\,q\in\sopm{Z}, u, v\in \alsbwo{Z}{\ordqd}, u\ordd v \right\}$$
  is convergent.\mlabel{item:drgb}

\item $S_\phi=\{\bws{[[u][v]]}-[N([u],[v])]\,|\,u, v\in \alsbwo{Z}{\ordqd}, u\ordd v\}$ is a \gsb in $\opliez$ with respect to the monomial order $\ordqd$. \mlabel{item:drgc}
\item The free $\phi$-Lie algebra on $Z$ is the free Lie algebra $\bfk\nlsbwd{Z}$ on $\Delta(Z)$ together with the operator $d$ on $\bfk\nlsbwd{Z}$.
Here $\Delta(Z)$ is given in Eq.~\meqref{eq:derf}, and $d$ on $\bfk\nlsbwd{Z}$ defined by the following recursion$:$

For any $[u]=[[u_1]\cdots [u_m]]\in \nlsbwd{Z}$ with $u_i\in \Delta(Z)$, $1\leq i \leq n$
\begin{enumerate}
  \item if $m=1$, then $[u]=\der{z}{n}\in \Delta(Z)$ for some $n\geq0$ and $z\in Z$, define $d([u]):=\der{z}{n+1}$;
  \item if $m>1$, then recursively define $d([u]):=[N([u_1],[u_2]\cdots [u_m])]$.
\end{enumerate}\mlabel{item:drgcc}
\item $S_\phi=\{\bws{uv}-N(u,v)\,|\,u,v\in\bfk\plie{Z}\}$ is a \gsb in $\bfk\plie{Z}$ with respect to the order $\ordqd$. \mlabel{item:drgd}
\item The term-rewriting system
  $$\Pia{S_\phi}=\left\{q\suba{\bws{u  v}}\rightarrow q\suba{N(u,v )}\,|\,q\in\sopm{Z}, u, v\in \bfk\plie{Z} \right \}$$
  is convergent.\mlabel{item:drge}
\item $\phi(x, y)= \bws{x y}- N(x,y )(\in \bfk\plie{\{x,y\}})$ is a differential type OPI. \mlabel{item:drgf}
\end{enumerate}
\mlabel{thm:gsbdt}
\end{theorem}
\begin{proof}
By Theorem~\mref{thm:alrcpeq}, Item~\meqref{item:drgb} $\Leftrightarrow$ Item~\meqref{item:drgc} $\Leftrightarrow$ Item~\meqref{item:drgd} $\Leftrightarrow$ Item~\meqref{item:drge}.
Further by~\cite[Theorem 5.7]{GSZ} and Remark~\mref{rem:none}, Item~\meqref{item:drge} $\Leftrightarrow$ Item~\meqref{item:drgf}.

Now we prove Item~\meqref{item:drga}$\Longleftrightarrow$ Item~\meqref{item:drgc}. Note that $\lbar{\phi(x , y )} = \lc xy\rc$ by our assumption. There are no intersection compositions and only three kinds of including compositions:
\begin{eqnarray*}
&&\big\langle \phi\big(q\suba{\bws{uv}} , w\big),\phi\big(u , v \big) \big\rangle_\bws{q\suba{\bws{uv}}w}\, \text{ for }\, u\ordd v,\, q\suba{\bws{uv}}\ordd w \in \alsbwo{Z}{\ordqd},\\
&&\big\langle \phi\big(u, q\suba{\bws{vw}}\big),\phi\big(v , w \big) \big\rangle_\bws{uq\suba{\bws{vw}}}\, \text{ for }\, v\ordd w,\, u\ordd q\suba{\bws{vw}} \in \alsbwo{Z}{\ordqd},\\
&&\big\langle \phi\big(uv , w\big),\phi\big(u ,vw \big) \big\rangle_\bws{uvw}\, \text{ for }\, u\ordd v\ordd w \in \alsbwo{Z}{\ordqd}.
\end{eqnarray*}
The trivialities of the first two intersection compositions follow from
\begin{eqnarray*}
&&\big\langle \phi\big(q\suba{\bws{uv}} , w\big),\phi\big(u , v \big) \big\rangle_\bws{q\suba{\bws{uv}}w}\\
&=&\phi\big(q\suba{\bws{uv}} , w\big)- [\bws{[[q\suba{\bws{uv}}][w]]}]_{\bws{uv}}\\
&=&-[N([q\suba{\bws{uv}}],[w] )]+[\bws{[[q\suba{N(u,v)}][w]]}]\\
&\equiv &-[N([q\suba{N(u,v)}],[w] )]+[N([q\suba{N(u,v)}],[w])]\\
&\equiv&0 \mod(S, \bws{q\suba{\bws{uv}}w} )
\end{eqnarray*}
and
\begin{eqnarray*}
&&\big\langle \phi\big(u, q\suba{\bws{vw}}\big),\phi\big(v , w \big) \big\rangle_\bws{uq\suba{\bws{vw}}}\\
&=&\phi\big(u , q\suba{\bws{vw}}\big)- [\bws{[[u][q\suba{\bws{vw}}]]}]_{\bws{vw}}\\
&=&-[N([u],[q\suba{\bws{vw}}] )]+[\bws{[[u][q\suba{N(v,w)}]]}]\\
&\equiv &-[N([u],[q\suba{N(v,w])}] )]+[N([u],[q\suba{N(v,w)}])]\\
&\equiv&0 \mod(S, \bws{uq\suba{\bws{vw}}} ).
\end{eqnarray*}
For the third one, we have
\begin{eqnarray*}
&&\langle \phi\big(uv , w\big) ,\phi\big(u ,vw \big) \rangle_\bws{uvw}\\
&=& [\phi\big(uv , w\big)]_{\bws{uv}} - [\phi\big(u ,vw \big)]_{\bws{vw}}\\
&=&\bws{[[u][w]][v]}-[N([[u][v]],[w])]+ [N([u],[[v][w]])].
\end{eqnarray*}
Then $S$ is a \gsb if and only if the composition
$$\big\langle \phi\big(uv , w\big),\phi\big(u ,vw \big) \big\rangle_\bws{uvw}=\bws{[[u][w]][v]}-[N([[u][v]],[w])]+ [N([u],[[v][w]])]$$
is trivial modulo $(S, \bws{uvw})$, that is
\begin{equation}
\bws{[[u][w]][v]}-[N([[u][v]],[w])]+ [N([u],[[v][w]])]=\sum_i \alpha_i[q_i| _{s_i}]_{\lbar{s_i}},\mlabel{eq:dirw}
\end{equation}
where each $ \alpha_i\in\bfk$, $q_i\in\sopm{Z}$, $s_i\in S$ and $[q_i\suba{s_i}]_{\lbar{s_i}}<_{dt}\bws{uvw}$.
Since $$\bws{[[u][w]][v]}-[N([[u][w]],[v])]=[\bws{uwv}]_{uw}=[\bws{\star v\suba{uw}}]_{\lbar{\phi{(u,w)}}},$$
and $\bws{uwv}<_{dt}\bws{uvw}$, Eq.~\meqref{eq:dirw} transforms into
\begin{eqnarray*}
&&\bws{[[u][w]][v]}-[N([[u][v]],[w])]+ [N([u],[[v][w]])]\\
&=&\Big([N([[u][w]],[v])]+[\bws{\star v\suba{uw}}]_{\lbar{\phi{(u,w)}}}\Big)-[N([[u][v]],[w])]+ [N([u],[[v][w]])]\\
&=&\sum_i \alpha_i[q_i| _{s_i}]_{\lbar{s_i}}.
\end{eqnarray*}
Notice that this condition is equivalent to that $$[N([[u][w]],[v])]-[N([[u][v]],[w])]+ [N([u],[[v][w]])]$$ is reduced to zero
by the rewriting system $\Pil{S}$ defined in Eq.~\meqref{eq:ltrs}, and is equivalent to that $\phi(x , y )$ is of differential type.

Item~\meqref{item:drgc}$\Longleftrightarrow$ Item~\meqref{item:drgcc} follows from Lemma~\mref{lem:cdmol} and $\irrl(S)=\nlsbwd{Z}$.
\end{proof}

\begin{remark}
In Theorem~\mref{thm:gsbdt}, the invariant monomial order can taken to be $\geq_{\rm dt}$~\cite{GSZ},
but can't to be $\ordqc$~\cite{QC}.
For example, for $\phi(u,v) = \lc [u v] \rc - [\lc u\rc v] - [u\lc v\rc]$ with $u\ordc v$,
we have $\lbar{\phi(u,v)} = \bws{u} v$, as $\bws{u} v\ordc u\bws{ v}\ordc \bws{u v}$.
\mlabel{re:last1}
\end{remark}

\begin{defn}
The {\bf operator degree} ${\rm deg}_P(u)$ of a monomial $u$ in $\oplie$ is the total number that the operator
$\bws{\, }$ appears in the monomial. The {\bf operator degree } of a polynomial $f$ in $\oplie$ is the maximum of the
operator degrees of the monomials appearing in $f$.
\mlabel{defn:olied}
\end{defn}

Next we try to find all differential type OLPIs under a restriction on the number of operators.
Since $\bfone$ is not in $\oplie$, we need to remove the cases for associative algebras involving $\bfone$ in Theorem~\mref{thm:cdto}
and propose the following answer to the Problem~\mref{pro:dp} in the case of operator degree  not exceeding two.

\begin{conjecture}
{\em (Classification of differential type OLPIs)}
Let $a , b , c , e\in\bfk$. Every expressions $[N ( x , y )]\in \opliex $ of differential type takes one (or more) of the forms below for $x> y$
\begin{enumerate}
  \item $ \ b \big( [x \bws{ y }]+[\bws{ x } y ]\big)+ c [\bws{ x }\bws{ y }]+ e[xy] $ where $b^2 = b + ce$,
  \item  $\ ce^2 [yx] + e[xy] + c [\bws{ y }\bws{ x }]- ce \big( [y \bws{ x }]+[\bws{ y } x] \big)$.
\end{enumerate}
\mlabel{thm:cdtol}
\end{conjecture}

By Theorems~\mref{thm:cdto} and~\mref{thm:gsbdt} and the fact that $\bfone\notin\oplie$, we have

\begin{corollary}
Let $[{N(x , y)}]\in \opliex \subseteq \bfk\plie{\{x,y\}}$ be from the list in Conjecture~\mref{thm:cdtol}. Then all the statements in Theorem~\mref{thm:gsbdt} hold.
\end{corollary}

\subsection{Rota-Baxter type OPIs and OLPIs}
Similar to the case of differential type, the classification of Rota-Baxter type OPIs was also studied.

\begin{defn}\mcite{ZGGS}
A {\bf Rota-Baxter type OPI}, defining a Rota-Baxter type operator, is
$$\phi(x , y ) := \bws{x}\bws{ y}- \bws{B(x , y )},$$
where
\begin{enumerate}
  \item $\bws{B(x, y )}$ is multi-linear in $x$ and $y$ ;
  \mlabel{item:ar}
  \item $\bws{B(x , y )}$ is a normal $\phi$-form, that is, $B(x , y )$  does not contain any subword of the form $\bws{u}\bws{v}$, for any $u,v\in \opmx$;
  \mlabel{item:br}
   \item The term-rewriting system $\Pi_\phi$ is terminating;
  \item For any set $Z$ with $u, v, w \in \opmzm$,
${B(B(u,v) , w )} - {B(u , B(v,w) )}\astarrow_{\Pia{\phi}} 0.$
  \mlabel{item:cr}
\end{enumerate}
\mlabel{defn:rbtyp}
\end{defn}
For Rota-Baxter type operators with low operator orders, we have 
\begin{theorem}\mcite{ZGGS} {\em (Classification of Rota-Baxter type operators)}
For any $c,\lambda\in\bfk$, $\phi := \bws{x}\bws{y} - \bws{B(x , y)}$, where ${B(x , y)}$  is taken from the list below, is of Rota-Baxter type.
\begin{enumerate}
  \item  $x\bws{ y }$ (average operator),
  \item  $\bws{ x } y$ (inverse average operator),
  \item  $ x \bws{ y }+ y \bws{ x },$
  \item  $\bws{ x } y +\bws{ y } x,$
  \item  $x \bws{ y }+\bws{ x } y -\bws{ xy }$ (Nijenhuis operator),
  \item  $x \bws{ y }+\bws{ x } y +\lambda xy$ (Rota-Baxter operator of weight $\lambda$ ),
  \item  $x \bws{ y }- x \bws{ \bfone } y +\lambda xy,$
  \item  $\bws{ x } y - x \bws{ \bfone } y +\lambda xy,$
  \item  $ x \bws{ y }+\bws{ x } y - x \bws{ \bfone } y +\lambda xy$ (generalized Leroux TD operator with weight $\lambda$ ),
  \item  $ x \bws{ y }+\bws{ x } y - xy \bws{ \bfone }- x \bws{ \bfone } y +\lambda xy, $
  \item  $ x \bws{ y }+\bws{ x } y - x \bws{ \bfone } y -\bws{ xy }+\lambda xy,$
  \item  $ x \bws{ y }+\bws{x } y - x \bws{ \bfone } y -\bws{ \bfone } xy +\lambda xy,$
  \item  $cx \bws{ \bfone } y +\lambda xy$ (generalized endomorphisms),
  \item  $cy \bws{ \bfone } x +\lambda yx$ (generalized antimorphisms).
\end{enumerate}
 \mlabel{thm:rbtyp}
\end{theorem}
This list of Rota-Baxter type operators may not be complete. See~\mcite{ZGGS} for more details.
Next we turn to Rota-Baxter type OLPIs.

\begin{defn}
\mlabel{defn:lrbtyp}
A {\bf Rota-Baxter type OLPI} is
a $$\phi(x , y ) :=[ \bws{x}\bws{ y}]- \bws{[B(x , y )]}\in\opliex$$
with $x> y$ such that
\begin{enumerate}
  \item ${[B(x , y )]}$ is multi-linear in $x$ and $y$ ;
  \mlabel{item:arrb}
  \item ${[B(x , y )]}$ is a normal $\phi$-form, that is, $[B(x , y )]$  does not contain the subword  $[\bws{[u]}\bws{[v]}]$ for any $u,v\in \alsbw{X}$ with $u\ord v$;
  \mlabel{item:brrb}
   \item The term-rewriting system
   $$\Pil{S_\phi}=\big\{[q\suba{\bws{u}\bws{  v}}]\rightarrow [R([q\suba{\phi(u , v )}]_{{\bws{u}\bws{  v}}})]\,|\,q\in\sopm{Z}, u, v\in \alsbwo{Z}{\ordq}, u\ord v \big \}$$ is terminating; \mlabel{item:drrb}
  \item  \mlabel{item:crrb}
For any well-ordered set $Z$ with $u\ord v\ord w \in \alsbw{Z}$,
$$[B([B([u],[w])],[v])]-[B([B([u],[v])],[w])]+ [B([u],[B([v],[w])])]\astarrow_{\Pil{S_\phi}} 0,$$
where $$\Pil{S_\phi}=\big\{[q\suba{\bws{u}\bws{  v}}]\rightarrow R([q\suba{\phi(u , v )}]_{{\bws{u}\bws{  v}}})\,|\,q\in\sopm{Z}, u, v\in \alsbwo{Z}{\ordq}, u\ord v \big \}.$$
\end{enumerate}
A linear operator on a Lie algebra that satisfies a Rota-Baxter type OLPI is called a {\bf Rota-Baxter type operator} on a Lie algebra.
\end{defn}

\begin{remark}
The condition~\meqref{item:cr} in Definition~\mref{defn:rbtyp} is to ensure $(\bws{u}\bws{v}) \bws{w} = \bws{u}(\bws{v} \bws{w})$;
while the above condition~\meqref{item:crrb} is for the consistence with the Jacobi identity
$$[[\bws{[u]}\bws{ [v]}] \bws{[w] } ]= [\bws{ [u]} [\bws{[v]}\bws{[w] }]] + [[\bws{[u]}\bws{[w]}] \bws{[v]}].$$
\end{remark}

\begin{problem} {\em (\rcp for Lie algebras: the Rota-Baxter case)} Find all OLPIs of Rota-Baxter
type by finding all expressions ${[B ( x , y )]}\in \opliex $ of Rota-Baxter type.
\mlabel{pro:rbp}
\end{problem}

We give some criteria for Rota-Baxter type OLPIs.

\begin{theorem} Let $Z$ be a well-ordered set.
Let
$$\phi(x , y ) = [\bws{x}\bws{ y}]- \bws{[B(x , y )]}\in \opliex \subseteq \bfk\plie{\{x,y\}},$$
with $x> y$ and $\phi(x,y)$ satisfies the conditions~\meqref{item:arrb} and~\meqref{item:brrb} in Definition~\mref{defn:lrbtyp}.
Let $\ordqb$ be an invariant monomial order satisfying $\lbar{\phi(u , v)}= \bws{u}\bws{ v}$ for $u\ordb v\in \alsbwo{Z}{\ordqb}$ $($such as $\ordqc$ in Example~\mref{ex:inva1}$)$.
The following statements are equivalent.
\begin{enumerate}
  \item $\phi(x, y) = [\bws{x}\bws{ y}]- \bws{[B(x , y )]}\in \opliex$ is a Rota-Baxter type OLPI.\mlabel{item:rdrga}
  \item The term-rewriting system
  $$\Pil{S_\phi}=\left\{\left .[q\suba{\bws{u}\bws{  v}}]\rightarrow R([q\suba{\phi(u , v )}]_{{\bws{u}\bws{  v}}})\,\right|\,q\in\sopm{Z}, u, v\in \alsbwo{Z}{\ordqb}, u\ordb v \right \}$$
  is convergent.\mlabel{item:rdrgb}
  \item The set $S_\phi=\left\{ \left .[\bws{[u]}\bws{ [v]}]- \bws{[B([u] , [v] )]}\,\right|\,u, v\in \alsbwo{Z}{\ordqb}, u\ordb v\right \}$ is a \gsb in $\opliez$ with respect to the order $\ordqb$.\mlabel{item:rdrgc}
  \item The set
  $$\irrl(S)\coloneqq \left\{\left .[w]\,\right|\, w\in\alsbwo{Z}{\ordqb}, w\neq q\suba{{{\bws{u}\bws{  v}}}}\,\text{ for }\, q\in\sopm{Z}, u, v\in \alsbwo{Z}{\ordqb}, u\ordb v\right\}$$
  is a linear basis of the free $\phi$-Lie algebra $\opliez/ \Idl(S).$\mlabel{item:rdrgcc}
  \item $S_\phi=\{  \bws{u}\bws{ v}- \bws{B(u , v )}\,|\,u,v\in\bfk\plie{Z} \}$ is a \gsb in $\bfk\plie{Z}$ with respect to the order $\ordqb$. \mlabel{item:rdrgd}
   \item The term-rewriting system
  $$\Pia{S_\phi}=\big\{q\suba{\bws{u  v}}\rightarrow q\suba{\bws{B(u,v )}}\,|\,q\in\sopm{Z}, u, v\in \bfk\plie{Z} \big \}$$
  is convergent.\mlabel{item:rdrge}
   \item $\phi(x, y) = \bws{x}\bws{ y}- \bws{B(x , y )}(\in \bfk\plie{\{x,y\}})$ is of Rota-Baxter type OPI. 
\mlabel{item:rdrgf}
\end{enumerate}
\mlabel{thm:gsbrbt}
\end{theorem}
\begin{proof}
Theorem~\mref{thm:alrcpeq} gives the equivalences  $$ \text{Item~\meqref{item:rdrgb}} \Leftrightarrow \text{Item~\meqref{item:rdrgc}} \Leftrightarrow \text{Item~\meqref{item:rdrgd}} \Leftrightarrow \text{Item~\meqref{item:rdrge}}.$$
Also \cite[Corollary 3.13]{ZGGS} yields
Item~\meqref{item:rdrge} $\Leftrightarrow$ Item~\meqref{item:rdrgf}. Further, Item~\meqref{item:rdrgc}$\Longleftrightarrow$ Item~\meqref{item:rdrgcc} follows from Lemma~\mref{lem:cdmol} and our assumption that $\lbar{\phi(u , v)}= \bws{u}\bws{ v}$ for $u\ordb v\in \alsbwo{Z}{\ordqb}$.

We finally prove Item~\meqref{item:drga}$\Longleftrightarrow$ Item~\meqref{item:drgc}.
As the proof is similar to the proof of Theorem~\mref{thm:gsbdt}, we only expose one type of including composition:
\begin{eqnarray*}
&&\langle \phi\big(uv , w\big) ,\phi\big(u ,vw \big) \rangle_{\bws{u}\bws{v}\bws{w}}\\
&=& [\phi\big(uv , w\big)]_{\bws{u}\bws{v}} - [\phi\big(u ,vw \big)]_{\bws{v}\bws{w}}\\
&=&[[\bws{[u]}\bws{[w]}]\bws{[v]}]-[B([B([u],[v])],[w])]+ [B([u],[B([v],[w])])]\\
&=&[B([B([u],[w])],[v])]-[B([B([u],[v])],[w])]+ [B([u],[B([v],[w])])]+[\star \bws{v}\suba{\bws{u}\bws{w}}]_{\lbar{\phi(u,w)}}
\end{eqnarray*}
for  $u\ordb v\ordb w \in \alsbwo{Z}{\ordqb}$ and $\bws{u}\bws{v}\bws{w}\ordb \bws{u}\bws{w}\bws{v}$.
\end{proof}

\begin{remark}
Note that in the above Theorem~\mref{thm:gsbrbt}~\meqref{item:rdrgcc}, the space $\bfk\irrl(S)$ can be view as the free $\phi$-Lie algebra with the operator $P:=\bws{~}$ and with the multiplication $\{-,-\}$ on $\bfk\irrl(S)$ defined on $\irrl(S)$ as follows and extended by bilinearity:

For any $[u],[v]\in \irrl(S)\subseteq\nlsbwd{Z}$,
\begin{enumerate}
  \item if either $u \in \nlsw{X}$ or $v \in  \nlsw{Z}$, then define $\{ [u][v]\} := [[u][v]]$ ;
  \item if $u = \bws{u'}$ and $v = \bws{v'}$ are both in $\bws{\irrl(S)}$, then define $\{ [u][v]\} := [\bws{B([u'],[v'])}]$ ;
  \item if $[u]=[[u_1]\cdots [u_m]], [v]=[[v_1]\cdots [v_n]]$ with $m>1$ or $n>1$, the  recursively define
  $$\{ [u][v]\} := [[u_1]\cdots [u_{m-1}]\{[u_m][v_1]\}\cdots [v_n]].$$
\end{enumerate}
\end{remark}

\begin{corollary}
Let $[{B(x , y)}]\in \opliex \subseteq \bfk\plie{\{x,y\}}$ be from the following list.
For any $\lambda\in\bfk$, the expressions $[{B(x , y)}]\in \opliex$ in the list below are of Rota-Baxter type, for $x> y$
\begin{enumerate}
	\item  $[x\bws{y}]$ (average operator),
	\item  $[\bws{ x } y]$ (inverse average operator),
	\item  $[\bws{ x } y]+[\bws{ y } x]$,
	\item  $[x\bws{y}]+[y\bws{ x }] $,
	\item  $[x\bws{y}]+[\bws{ x } y] -\bws{ [xy] }$ (Nijenhuis operator),
	\item  $[\bws{ x } y]+[x\bws{y}] +\lambda [xy]$ (Rota-Baxter operator of weight $\lambda$).
\end{enumerate}
Then any (and hence all) the statements in Theorem~\mref{thm:gsbrbt} hold.
\end{corollary}

\begin{proof}
Note that the list in the corollary is obtained from the list in Theorem~\mref{thm:rbtyp} by only taking the ones without $\bfone$. 
Thus the corollary follows from Theorems~\mref{thm:rbtyp} and~\mref{thm:gsbrbt}, and the fact that $\bfone$ is not in $\oplie$.
\end{proof}

A natural question is to determine whether these are all the Rota-Baxter type OPLIs. 

\subsection{Modified Rota-Baxter OLPI}
To give applications beyond the differential type and Rota-Baxter type OLPIs, we consider the {\bf modified Rota-Baxter OLPI} of weight $\lambda$, defined to be
$$\phi(x,y)=[ \bws{x}\bws{ y}]-\bws{[x \bws{ y }]}-\bws{[\bws{ x } y]}-\lambda [xy] \in\opliex,$$
with $x> y$.
When $\lambda=-1$, the modified Rota-Baxter Lie algebra  is a Lie algebra $L$ equipped with a linear map $P : L \to L$ satisfying
$$[P(x)P(y)]=P([P(x)y])+P([xP(y)])-[xy]\,\text{ for }\,x,y\in L,$$
called the {\bf modified Yang-Baxter equation}~\mcite{Bo,Kup,STS}.

\begin{theorem}
 Let $$\phi(x,y)=[ \bws{x}\bws{ y}]-\bws{[x \bws{ y }]}-\bws{[\bws{ x } y]}-\lambda [xy] \in\opliex$$
be an OLPI with $x> y$.
Let $Z$ be a well-ordered set. Let $\ordqb$ be an invariant monomial order $\ordqb$ on $\plie{Z}$  satisfying $\lbar{\phi(u , v)}= \bws{u}\bws{ v}$ for $u\ordb v\in \alsbwo{Z}{\ordqb}$ $($such as $\ordqc$~\cite{QC} in Example.~\mref{ex:inva1}$)$.
The following equivalent statements hold.
\begin{enumerate}
  \item $S_\phi=\left \{\left .[\bws{[u]}\bws{ [v]}]-\bws{[[u] \bws{ [v] }]}-\bws{[\bws{ [u] } [v]]}-\lambda [[u][v]] \,\right|\,\,u, v\in \alsbwo{Z}{\ordqb}, u\ordb v\right \}$ is a \gsb in $\opliez$ with respect to the monomial order $\ordqb$.
  \item The term-rewriting system $$\Pil{S_\phi}=\big\{[q\suba{\bws{u}\bws{  v}}]\rightarrow R([q\suba{\phi(u,v)}]_{{\bws{u}\bws{v}}})\,|\,q\in\sopm{Z}, u, v\in \alsbwo{Z}{\ordqb}, u\ordb v \big \}$$
  is convergent.
  \item $S_\phi=\left\{\left .\bws{u}\bws{ v}-\bws{u \bws{ v }}-\bws{\bws{ u } v}-\lambda uv\,\right|\,u,v\in\plie{Z}\right\}$ is a \gsb in $\bfk\plie{Z}$ with respect to the monomial order $\ordqb$.
  \mlabel{it:modrb3}
  \item
  The term-rewriting system $$\Pia{S_\phi}=\big\{q\suba{\bws{u}\bws{  v}}\rightarrow q\suba{\bws{\bws{u}  v}+\bws{{u}\bws{  v}}+\lambda uv}\,|\,q\in\sopm{Z}, u, v\in \plie{Z}\big\}$$
  is convergent.
\end{enumerate}
\mlabel{thm:modrbDL}
\end{theorem}
\begin{proof}
The equivalence of the four statements follows from Theorem~\mref{thm:alrcpeq}. By~\cite[Theorem~4.6]{GG}, statement~\meqref{it:modrb3} and hence the other statements hold. 
\end{proof}

\smallskip
\noindent
{\bf Acknowledgments.}
This work is supported by the National Natural Science Foundation of
China (Grant Nos. 12071191, 11861051), the Natural Science Foundation of Gansu
Province (Grant No. 20JR5RA249) and the Natural Science Foundation of Shandong Province
(ZR2020MA002).

\end{document}